%% file: ms.tex
\documentclass{article}
\usepackage{fullpage}

\usepackage[parfill]{parskip}

\usepackage[utf8]{inputenc}
\usepackage[T1]{fontenc}

\usepackage[final]{microtype}
\microtypecontext{spacing=nonfrench}

\usepackage{enumitem}
\setlist[description]{font=\bfseries}
\usepackage{amsmath,amsfonts,amssymb,mathtools,amsthm}
\newtheorem{theorem}{Theorem}[section]
\newtheorem{lemma}[theorem]{Lemma}
\usepackage{stmaryrd}
\usepackage{siunitx}
\usepackage{booktabs,multirow,multicol,arydshln}
\usepackage{algpseudocode}

\usepackage{xcolor}
\usepackage{natbib}
\usepackage{xurl}
\usepackage[breaklinks]{hyperref}
\usepackage[noabbrev,capitalise]{cleveref}
\colorlet{inlinkcolor}{green!50!black}
\colorlet{exlinkcolor}{red!50!black}
\hypersetup{
colorlinks = true,
allcolors = inlinkcolor,
urlcolor = exlinkcolor,
}
\crefname{equation}{}{}

\usepackage[section]{placeins}
\usepackage{float}
\usepackage{tikz}
\usepackage{pgfplots}
\pgfplotsset{compat=1.16}
\usepgfplotslibrary{groupplots}
\usepgfplotslibrary{statistics}

\input{defs}

\usepackage{authblk}

\begin{document}

\title{Performance enhancements for a generic \\conic interior point algorithm}

\author[1]{Chris Coey}
\author[1]{Lea Kapelevich}
\author[2]{Juan Pablo Vielma}
\affil[1]{Operations Research Center, MIT, Cambridge, MA}
\affil[2]{Google Research and MIT Sloan School of Management, Cambridge, MA}

\maketitle

\begin{abstract}
\input{abstract}
\end{abstract}

\paragraph{Funding}
This work has been partially funded by the National Science Foundation under grant OAC-1835443 and the Office of Naval Research under grant N00014-18-1-2079.

\setcounter{tocdepth}{3}
\tableofcontents

\input{intro}
\input{prelim}

\input{exotic}
\input{form}
\input{alg}

\input{cones}
\input{testing}

\appendix
\input{linalg}
\input{search}

\input{toos}

\bibliographystyle{abbrvnat}
\bibliography{refs}

\end{document}

%% file: defs.tex
\DeclareMathOperator{\cl}{cl}

\mathchardef\mhyphen="2D

\DeclareMathOperator{\spec}{spec}
\DeclareMathOperator{\lspec}{\ell_{\spec}}
\DeclareMathOperator{\linf}{\ell_{\infty}}
\DeclareMathOperator{\ltwo}{\ell_2}
\DeclareMathOperator{\sepspec}{sepspec}
\DeclareMathOperator{\logdet}{logdet}
\DeclareMathOperator{\rtdet}{rtdet}
\DeclareMathOperator{\sqr}{sqr}
\DeclareMathOperator{\matsqr}{matsqr}

\DeclareMathOperator{\relentr}{relent}
\DeclareMathOperator{\matrelentr}{matrelent}
\DeclareMathOperator{\geom}{geo}
\DeclareMathOperator{\power}{pow}
\DeclareMathOperator{\gpower}{gpow}

\DeclareMathOperator{\lmi}{LMI}
\DeclareMathOperator{\sppsd}{sPSD}
\DeclareMathOperator{\dnn}{DNN}
\DeclareMathOperator{\wsos}{SOS}
\DeclareMathOperator{\matwsos}{matSOS}

\newcommand{\knn}{\K_{\geq}}
\newcommand{\kpsd}{\K_{\succeq}}
\newcommand{\klinf}{\K_{\linf}}
\newcommand{\kltwo}{\K_{\ltwo}}
\newcommand{\klspec}{\K_{\lspec}}
\newcommand{\ksqr}{\K_{\sqr}}
\newcommand{\klog}{\K_{\log}}
\newcommand{\klogdet}{\K_{\logdet}}
\newcommand{\krtdet}{\K_{\rtdet}}
\newcommand{\kmatsqr}{\K_{\matsqr}}

\newcommand{\krelentr}{\K_{\relentr}}
\newcommand{\kmatrelentr}{\K_{\matrelentr}}
\newcommand{\kgeom}{\K_{\geom}}
\newcommand{\kpower}{\K_{\power}}
\newcommand{\kgpower}{\K_{\gpower}}
\newcommand{\klmi}{\K_{\lmi}}
\newcommand{\ksepspec}{\K_{\sepspec}}
\newcommand{\ksppsd}{\K_{\sppsd}}
\newcommand{\kdnn}{\K_{\dnn}}
\newcommand{\kwsos}{\K_{\wsos}}
\newcommand{\kmatwsos}{\K_{\matwsos}}
\newcommand{\kltwowsos}{\K_{\ell_2 \! \wsos}}
\newcommand{\klonewsos}{\K_{\ell_1 \! \wsos}}

\DeclareMathOperator{\diag}{diag}
\DeclareMathOperator{\Diag}{Diag}
\DeclareMathOperator{\tr}{tr}
\DeclareMathOperator{\intr}{int}
\DeclareMathOperator{\mat}{mat}
\DeclareMathOperator{\vect}{vec}
\DeclareMathOperator{\sdim}{sd}
\DeclareMathOperator*{\argmin}{arg\,min}

\DeclareMathOperator{\ch}{ch}

\newcommand{\iin}[1]{\llbracket #1 \rrbracket}
\newcommand{\tsum}[1]{{\textstyle \sum_{#1}}}
\newcommand{\tprod}[1]{{\textstyle \prod_{#1}}}
\newcommand{\norm}[1]{\lVert #1 \rVert}

\newcommand{\bbR}{\mathbb{R}}

\newcommand{\bbS}{\mathbb{S}}
\newcommand{\cQ}{\mathcal{Q}}
\newcommand{\C}{\mathcal{C}}
\newcommand{\K}{\mathcal{K}}
\newcommand{\bK}{\bar{\mathcal{K}}}
\newcommand{\bKn}{\bar{K}}
\newcommand{\Kpr}{K_\text{pr}}
\newcommand{\Kdu}{K_\text{du}}

\newcommand{\pitwo}{\pi_{\ell_2}}
\newcommand{\piinf}{\pi_{\ell_\infty}}
\newcommand{\bigO}{\mathcal{O}}

\newcommand{\Tau}{\mathrm{T}}
\newcommand{\bs}{\bar{s}}
\newcommand{\bz}{\bar{z}}

\newcommand{\dotD}{\dot{D}}
\newcommand{\ddotD}{\ddot{D}}
\newcommand{\dotL}{\dot{L}}
\newcommand{\ddotL}{\ddot{L}}
\newcommand{\dotU}{\dot{U}}
\newcommand{\ddotU}{\ddot{U}}
\newcommand{\dotR}{\dot{R}}
\newcommand{\ddotR}{\ddot{R}}

%% file: abstract.tex
In recent work, we provide computational arguments for expanding the class of proper cones recognized by conic optimization solvers, to permit simpler, smaller, more natural conic formulations.
We define an exotic cone as a proper cone for which we can implement a small set of tractable (i.e.\ fast, numerically stable, analytic) oracles for a logarithmically homogeneous self-concordant barrier for the cone or for its dual cone.
Our extensible, open source conic interior point solver, Hypatia, allows modeling and solving any conic optimization problem over a Cartesian product of exotic cones.
In this paper, we introduce Hypatia's interior point algorithm.
Our algorithm is based on that of \citet{skajaa2015homogeneous}, which we generalize by handling exotic cones without tractable primal oracles.
With the goal of improving iteration count and solve time in practice, we propose a sequence of four enhancements to the interior point stepping procedure of \citet{skajaa2015homogeneous}, which alternates between prediction and centering steps: (1) we use a less restrictive central path proximity condition, (2) we adjust the prediction and centering directions using a new third order directional derivative barrier oracle, (3) we use a single backtracking search on a quadratic curve instead of two line searches, and (4) we use a combined prediction and centering step.
We implement 23 useful exotic cones in Hypatia.
We summarize the complexity of computing oracles for these cones, showing that our new third order oracle is not a bottleneck, and we derive efficient and numerically stable oracle implementations for several cones.
We generate a diverse benchmark set of 379 conic problems from 37 different applied examples.
Our computational testing shows that each stepping enhancement improves Hypatia's iteration count and solve time.
Altogether, the enhancements reduce the shifted geometric means of iteration count and solve time by over 80\% and 70\% respectively.

%% file: intro.tex
\section{Introduction}
\label{sec:intro}

Any convex optimization problem may be represented as a conic problem that minimizes a linear function over the intersection of an affine subspace with a Cartesian product of primitive proper cones (i.e.\ irreducible, closed, convex, pointed, and full-dimensional conic sets). 
Under certain conditions, a conic problem has a simple and easily checkable certificate of optimality, primal infeasibility, or dual infeasibility \citep{permenter2017solving}.
Most commercial and open-source conic solvers (such as CSDP \citep{borchers1999csdp}, CVXOPT \citep{andersen2011interior}, ECOS \citep{serrano2015algorithms}, MOSEK \citep{mosek2020modeling}, SDPA \citep{yamashita2003implementation}, Alfonso \citep{papp2017homogeneous}) implement primal-dual interior point methods (PDIPMs) based on the theory of logarithmically homogeneous self-concordant barrier (LHSCB) functions.
Compared to first order conic methods (see \citet{o2016conic} on SCS solver), idealized PDIPMs typically exhibit higher per-iteration cost, but have a lower iteration complexity of $\bigO(\sqrt{\nu} \log (1 / \varepsilon))$ iterations to converge to $\varepsilon$ tolerance, where $\nu$ is the barrier parameter of the LHSCB.
We limit the scope of this paper to conic PDIPMs, but note that there are other notions of duality and PDIPMs for convex problems outside of the conic realm (see e.g.\ \citet{karimi2018primal}).

\subsection{Conic optimization with Hypatia solver}
\label{sec:intro:exotic}

Hypatia \citep{coey2020solving} is an open-source, extensible conic primal-dual interior point solver.\footnote{
Hypatia is available at \href{https://github.com/chriscoey/Hypatia.jl}{github.com/chriscoey/Hypatia.jl} under the MIT license.
See \citet{coey2021hypatiab} for documentation, examples, and instructions for using Hypatia.
}
Hypatia is written in the Julia language \citep{bezanson2017julia} and is accessible through a flexible, low-level native interface or the modeling tool JuMP \citep{dunning2017jump}.
A key feature of Hypatia is a generic cone interface that allows users to define new \emph{exotic cones}.
An exotic cone is a proper cone for which we can implement a small set of tractable LHSCB oracles (listed in \cref{sec:intro:sy}) for either the cone or its dual cone.
Defining a new cone through Hypatia's cone interface makes both the cone and its dual available for use in conic formulations.
We have already predefined 23 useful exotic cone types (some with multiple variants) in Hypatia.
Several cones are new and required the development of LHSCBs and efficient procedures for oracle evaluations (see \citet{coey2021self,kapelevich2021sum}).

Advanced conic solvers such as MOSEK 9 currently recognize at most only a handful of standard cones: nonnegative, (rotated) second order, positive semidefinite (PSD), and three-dimensional exponential and power cones.
In \citet{coey2020solving}, we show across seven applied examples that modeling with the much larger class of exotic cones often permits simpler, smaller, more natural conic formulations.
Our computational experiments with these examples demonstrate the potential advantages, especially in terms of solve time and memory usage, of solving natural formulations with Hypatia compared to solving standard conic extended formulations with either Hypatia or MOSEK 9.
However, a description of Hypatia's algorithms was outside the scope of \citet{coey2020solving}.
In contrast, the main purpose of this paper is to introduce some key features of our exotic conic PDIPM that enable it to be both general and performant.

\subsection{The Skajaa-Ye algorithm}
\label{sec:intro:sy}

Most conic PDIPM solvers use efficient algorithms specialized for symmetric cones, in particular, the nonnegative, (rotated) second order, and PSD cones.
Although non-symmetric conic PDIPMs proposed by \citet{nesterov1996infeasible,nesterov2012towards} can handle a broader class of cones, they have several disadvantages compared to specialized symmetric methods (e.g. requiring larger linear systems, strict feasibility of the initial iterate, or conjugate LHSCB oracles).
The algorithm by \citet{skajaa2015homogeneous}, henceforth referred to as \emph{SY}, addresses these issues by approximately tracing the central path of the homogeneous self-dual embedding (HSDE) \citep{andersen2003implementing,xu1996simplified} to an approximate solution for the HSDE.
This final iterate provides an approximate conic certificate for the conic problem, if a conic certificate exists.
The SY algorithm relies on an idea by \citet{nesterov2012towards} that a high quality prediction direction (enabling a long step and rapid progress towards a solution) can be computed if the current iterate is in close proximity to the central path (i.e.\ it is an approximate scaling point).
To restore centrality after each prediction step, SY performs a series of centering steps.

By using a different definition of central path proximity to \citet{nesterov2012towards}, SY avoids needing conjugate LHSCB oracles.\footnote{
Some proposed techniques such as the Hessian scaling updates and central path proximity definitions of \citet{myklebust2014interior,dahl2021primal} require conjugate LHSCB oracles.
}
Indeed, a major advantage of SY is that it only requires access to a few tractable oracles for the primal cone: an initial interior point, feasibility check, and gradient and Hessian evaluations for the LHSCB.
In our experience, for a large class of proper cones, these oracles can be evaluated analytically, i.e.\ without requiring the implementation of iterative numerical procedures (such as optimization) that can be expensive and may need numerical tuning.
Conjugate LHSCB oracles in general require optimization, and compared to the analytic oracles, they are often significantly less efficient and more prone to numerical instability.

\subsection{Practical algorithmic developments}
\label{sec:intro:alg}

For many proper cones of interest, including most of Hypatia's non-symmetric cones, we are aware of LHSCBs with tractable oracles for either the cone or its dual cone but not both.
Suppose a problem involves a Cartesian product of exotic cones, some with primal oracles implemented and some with dual oracles implemented (as in several example formulations described in \citet{coey2020solving}). 
In this case, SY can solve neither the primal problem nor its conic dual, as SY requires primal oracles.
Our algorithm generalizes SY to allow a conic formulation over any Cartesian product of exotic cones.

The focus of \citet{skajaa2015homogeneous} is demonstrating that SY has the best known iteration complexity for conic PDIPMs.
This complexity analysis was corrected by \citet{papp2017homogeneous}, who implemented SY in their recent MATLAB solver Alfonso \citep{papp2020alfonso,papp2021alfonso}.
It is well known that performant PDIPM implementations tend to violate assumptions used in iteration complexity analysis, so in this paper we are not concerned with iteration complexity.
Our goal is to reduce iteration counts and solve times in practice, by enhancing the performance of the interior point stepping procedure proposed by SY and implemented by Alfonso.

The basic SY-like stepping procedure computes a prediction or centering direction by solving a large structured linear system, performs a backtracking line search in the direction, and steps as far as possible given a restrictive central path proximity condition.
We propose a sequence of four practical performance enhancements.
\begin{description}
\item[Less restrictive proximity.]
We use a relaxed central path proximity condition, allowing longer prediction steps and fewer centering steps.
\item[Third order adjustments.]
After computing the prediction or centering direction, we compute a \emph{third order adjustment} (TOA) direction using a new \emph{third order oracle} (TOO) for exotic cones.
We use a line search in the unadjusted direction to determine how to combine it with the TOA direction, before performing a second line search and stepping in the new adjusted direction.
\item[Curve search.]
Due to the central path proximity checks, each backtracking line search can be quite expensive.
Instead of performing two line searches, we use a single backtracking search along a particular quadratic curve of combinations of the unadjusted and TOA directions.
\item[Combined directions.]
Unlike SY, most conic PDIPMs do not use separate prediction and centering phases.
We compute the prediction and centering directions and their associated TOA directions, then perform a backtracking search along a quadratic curve of combinations of all four directions.
\end{description}

Our TOA approach is distinct from the techniques by \citet{mehrotra1992implementation,dahl2021primal} that also use higher order LHSCB information.\footnote{
To avoid confusion, we do not use the term `corrector' in this paper.
In the terminology of \citet{mehrotra1992implementation,dahl2021primal} our TOA approach is a type of `higher order corrector' technique, but also our unadjusted centering direction is referred to by \citet{skajaa2015homogeneous,papp2017homogeneous} as the `corrector' direction.
}
Unlike these techniques, we derive adjustments (using the TOO) for both the prediction and centering directions.
Our TOO has a simpler and more symmetric structure than the third order term used by \citet{dahl2021primal}, and we leverage this for fast and numerically stable evaluations.
Whereas the method by \citet{mehrotra1992implementation} only applies to symmetric cones, and \citet{dahl2021primal} test their technique only for the standard exponential cone, we implement and test our TOO for all of Hypatia's 23 predefined cones.
In our experience, requiring a tractable TOO is only as restrictive as requiring tractable gradient and Hessian oracles.
We show that the time complexity of the TOO is no higher than that of the other required oracles for each of our cones.
To illustrate, we describe efficient and numerically stable TOO procedures for several cones that can be characterized as intersections of slices of the PSD cone.

Although this paper is mainly concerned with the stepping procedures, we also outline our implementations of other key algorithmic components. 
These include preprocessing of problem data, finding an initial iterate, the solution of structured linear systems for search directions, and efficient backtracking searches with central path proximity checks. 
We note that Hypatia has a variety of algorithmic options for these components; these different options can have a dramatic impact on overall solve time and memory usage, but in most cases they have minimal effect on the iteration count.
For the purposes of this paper, we only describe and test one set of (default) options for these components.

\subsection{Benchmark instances and computational testing}
\label{sec:intro:testing}

We implement and briefly describe 37 applied examples (available in Hypatia's examples folder), each of which has options for creating formulations of different types and sizes.
From these examples, we generate 379 problem instances of a wide range of sizes.
Since there is currently no conic benchmark storage format that recognizes more than a handful of cone types, we generate all instances on the fly using JuMP or Hypatia's native interface.
All of Hypatia's predefined cones are represented in these instances, so we believe this is the most diverse conic benchmark set available.

On this benchmark set, we run five different stepping procedures: the basic SY-like procedure (similar to Alfonso) and the sequence of four cumulative enhancements to this procedure.
Our results show that each enhancement tends to improve Hypatia's iteration count and solve time, with minimal impact on the number of instances solved.
We do not enforce time or iteration limits, but we note that under strict limits the enhancements would greatly improve the number of instances solved.
The TOA enhancement alone leads to a particularly consistent improvement of around 45\% for iteration counts.
Overall, the enhancements together reduce the iterations and solve time by more than 80\% and 70\% respectively.
For instances that take more iterations or solve time, the enhancements tend to yield greater relative improvements in these measures.

\subsection{Overview}
\label{sec:intro:over}

In \cref{sec:prelim}, we define our mathematical notation.
In \cref{sec:exotic}, we define exotic cones, LHSCBs, and our required cone oracles (including the TOO).
In \cref{sec:form}, we describe Hypatia's general primal-dual conic form, associated conic certificates, and the HSDE.
In \cref{sec:alg}, we define the central path of the HSDE and central path proximity measures, and we outline Hypatia's high level algorithm.
We also derive the prediction and centering directions and our new TOA directions, and we describe the SY-like stepping procedure and our series of four enhancements to this procedure.
In \crefrange{sec:linalg}{sec:search}, we discuss advanced procedures for preprocessing and initial point finding, solving structured linear systems for directions, and performing efficient backtracking searches and proximity checks.
In \cref{sec:cones}, we briefly introduce Hypatia's predefined exotic cones and show that our TOO is relatively cheap to compute, and in \cref{sec:toos} we describe some TOO implementations.
In \cref{sec:testing}, we summarize our applied examples and exotic conic benchmark instances, and finally we present our computational results demonstrating the practical efficacy of our stepping enhancements.

%% file: prelim.tex
\section{Notation}
\label{sec:prelim}

For a natural number $d$, we define the index set $\iin{d} \coloneqq \{1, 2, \ldots, d\}$.
Often we construct vectors with round parentheses, e.g. $(a, b, c)$, and matrices with square brackets, e.g. $\begin{bsmallmatrix} a & b \\ c & d \end{bsmallmatrix}$.
For a set $\C$, $\cl(\C)$ and $\intr(\C)$ denote the closure and interior of $\C$, respectively.

$\bbR$ denotes the space of reals, and $\bbR_{\geq}$, $\bbR_>$, $\bbR_{\leq}$, $\bbR_<$ denote the nonnegative, positive, nonpositive, and negative reals.
$\bbR^d$ is the space of $d$-dimensional real vectors, and $\bbR^{d_1 \times d_2}$ is the $d_1$-by-$d_2$-dimensional real matrices.
The vectorization operator $\vect : \bbR^{d_1 \times d_2} \to \bbR^{d_1 d_2}$ maps matrices to vectors by stacking columns, and its inverse operator is $\mat_{d_1, d_2} : \bbR^{d_1 d_2} \to \bbR^{d_1 \times d_2}$.

$\bbS^d$ is the space of symmetric matrices with side dimension $d$, and $\bbS^d_{\succeq}$ and $\bbS^d_{\succ}$ denote the positive semidefinite and positive definite symmetric matrices. 
The inequality $S \succeq Z$ is equivalent to $S - Z \in \bbS^d_{\succeq}$ (and similarly for the strict inequality $\succ$ and $\bbS^d_{\succ}$). 
We let $\sdim(d) \coloneqq d (d + 1) / 2$ be the dimension of the vectorized upper triangle of $\bbS^d$.
We overload the vectorization operator $\vect : \bbS^d \to \bbR^{\sdim(d)}$ to perform an \emph{svec} transformation, which rescales off-diagonal elements by $\sqrt{2}$ and stacks columns of the upper triangle (or equivalently, rows of the lower triangle). 
For example, for $S \in \bbS^3$ we have $\sdim(3) = 6$ and $\vect(S) = ( S_{1,1}, \sqrt{2} S_{1,2}, S_{2,2}, \sqrt{2} S_{1,3}, \sqrt{2} S_{2,3}, S_{3,3} ) \in \bbR^{\sdim(3)}$.
The inverse mapping $\mat : \bbR^{\sdim(d)} \to \bbS^d$ is well-defined.

For a vector or matrix $A$, the transpose is $A'$ and the trace is $\tr(A)$.
We use the standard inner product on $\bbR^d$, i.e.\ $s' z = \tsum{i \in \iin{d}} s_i z_i$ for $s, z \in \bbR^d$, which equips $\bbR^d$ with the standard norm $\lVert s \rVert = (s' s)^{1/2}$.
The linear operators $\vect$ and $\mat$ preserve inner products, e.g. $\vect(S)' \vect(Z) = \tr(S' Z)$ for $S, Z \in \bbR^{d_1 \times d_2}$ or $S, Z \in \bbS^d$.
$\Diag : \bbR^d \to \bbS^d$ is the diagonal matrix of a given vector, and $\diag : \bbS^d \to \bbR^d$ is the vector of the diagonal of a given matrix. 
For dimensions implied by context, $e$ is a vector of $1$s, $e_i$ is the $i$th unit vector, and $0$ is a vector or matrix of $0$s.

$\lvert x \rvert$ is the absolute value of $x \in \bbR$ and $\log(x)$ is the natural logarithm of $x > 0$.
$\det(X)$ is the determinant of $X \in \bbS^d$, and $\logdet(X)$ is the log-determinant of $X \succ 0$.
For a vector $x \in \bbR^d$, $\lVert x \rVert_\infty = \max_{i \in \iin{d}} \lvert x_i \rvert$ is the $\ell_\infty$ norm and $\lVert x \rVert_1 = \sum_{i \in \iin{d}} \lvert x_i \rvert$ is the $\ell_1$ norm.

Suppose the function $f : \intr(\C) \rightarrow \bbR$ is strictly convex and three times continuously differentiable on the interior of a set $\C \subset \bbR^d$.
For a point $p \in \intr(\C)$, we denote the gradient and Hessian of $f$ at $p$ as $\nabla f(p) \in \bbR^d$ and $\nabla^2 f(p) \in \bbS^d_{\succ}$. 
Given an $h \in \bbR^d$, the first, second, and third order directional derivatives of $f$ at $p$ in direction $h$ are $\nabla f(p) [h] \in \bbR$, $\nabla^2 f(p) [h, h] \in \bbR_{\geq}$, and $\nabla^3 f(p) [h, h, h] \in \bbR$.

%% file: exotic.tex
\section{Exotic cones and oracles}
\label{sec:exotic}

Let $\K$ be a proper cone in $\bbR^q$, i.e.\ a conic subset of $\bbR^q$ that is closed, convex, pointed, and full-dimensional (see \citet{skajaa2015homogeneous}).
Note that requiring $\K$ to be a subset of $\bbR^q$ simplifies our notation but is not restrictive, e.g.\ for the PSD cone, we use the standard svec vectorization (see \cref{sec:prelim}).
The dual cone of $\K$ is $\K^\ast$, which is also a proper cone in $\bbR^q$:
\begin{equation}
\K^\ast \coloneqq \{ 
z \in \bbR^q : s' z \geq 0, \forall s \in \K 
\}.
\label{eq:Kdual}
\end{equation}

Following \citet[Sections 2.3.1 and 2.3.3]{nesterov1994interior}, $f : \intr(\K) \to \bbR$ is a $\nu$-LHSCB for $\K$, where $\nu \geq 1$ is the \emph{LHSCB parameter}, if it is three times continuously differentiable, strictly convex, satisfies $f(s_i) \to \infty$ along every sequence $s_i \in \intr(\K)$ converging to the boundary of $\K$, and:
\begin{subequations}
\begin{alignat}{3}
\big\lvert \nabla^3 f(s) [h, h, h] \big\rvert 
& \leq 2 \bigl( \nabla^2 f(s) [h, h] \bigr)^{3/2} 
&\quad&\forall s \in \K, h \in \bbR^q,
\label{eq:lhscb:sc}
\\
f(\theta s) 
& = f(s) - \nu \log(\theta) 
&\quad&\forall s \in \K, \theta \in \bbR.
\label{eq:lhscb:lh}
\end{alignat}
\end{subequations}
Following \citet[Section 3.3]{renegar2001mathematical}, we define the \emph{conjugate} of $f$, $f^\ast : \intr (\K^\ast) \to \bbR$, as: 
\begin{equation}
f^\ast(z) \coloneqq 
-\textstyle\inf_{s \in \intr(\K)} \{ s' z + f(s) \},
\label{eq:conj}
\end{equation}
which is a $\nu$-LHSCB for $\K^\ast$.

A Cartesian product $\K = \K_1 \times \cdots \times \K_K$ of $K$ proper cones is a proper cone, and its dual cone is $\K^\ast = \K_1^\ast \times \cdots \times \K_K^\ast$.
In this case, if $f_k$ is a $\nu_k$-LHSCB for $\K_k$, then $\sum_{k \in \iin{K}} f_k$ is an LHSCB for $\K$ with parameter $\sum_{k \in \iin{K}} \nu_k$ \citep[Proposition 2.3.3]{nesterov1994interior}.
We call $\K$ a primitive cone if it cannot be written as a Cartesian product of two or more lower-dimensional cones (i.e.\ $K$ must equal $1$).
Note $\K^\ast$ is primitive if and only if $\K$ is primitive.
Primitive proper cones are the fundamental building blocks of conic formulations.

We call a proper cone $\K$ an exotic cone if we can implement a particular set of tractable oracles for either $\K$ or $\K^\ast$.
Suppose we have tractable oracles for $\K \subset \bbR^q$ and let $f : \intr(\K) \to \bbR$ denote the $\nu$-LHSCB for $\K$.
The oracles for $\K$ that we require in this paper are as follows.
\begin{description}
\item[Feasibility check.] 
The strict feasibility oracle checks whether a given point $s \in \bbR^q$ satisfies $s \in \intr(\K)$.
\item[Gradient and Hessian evaluations.]
Given a point $s \in \intr(\K)$, the gradient oracle $g$ and Hessian oracle $H$ evaluated at $s$ are:
\begin{subequations}
\begin{align}
g(s) & \coloneqq \nabla f(s) \in \bbR^q,
\\
H(s) & \coloneqq \nabla^2 f(s) \in \bbS^q_{\succ}.
\end{align}
\end{subequations}
\item[Third order directional derivative.] 
Given a point $s \in \intr(\K)$ and a direction $\delta_s \in \bbR^q$, our new \emph{third order oracle} (TOO), denoted $\Tau$, is a rescaled third order directional derivative vector:
\begin{equation}
\Tau (s, \delta_s) 
\coloneqq -\tfrac{1}{2} \nabla^3 f(s) [\delta_s, \delta_s] 
\in \bbR^q.
\label{eq:too}
\end{equation}
\item[Initial interior point.] 
The initial interior point $t \in \intr(\K)$ is an arbitrary point in the interior of $\K$ (which is nonempty since $\K$ is proper). 
\end{description}

In \cref{sec:cones}, we introduce Hypatia's predefined cones and discuss the time complexity of computing the feasibility check, gradient, Hessian, and TOO oracles.
In \cref{sec:toos}, we describe efficient and numerically stable techniques for computing these oracles for a handful of our cones.
Although Hypatia's generic cone interface allows specifying additional oracles that can improve speed and numerical performance (e.g.\ a dual cone feasibility check, Hessian product, and inverse Hessian product), these optional oracles are outside the scope of this paper.

For the initial interior point (which Hypatia only calls once, when finding an initial iterate), we prefer to use the \emph{central point} of $\K$.
This is the unique point satisfying $t \in \intr (\K) \cap \intr (\K^\ast)$ and $t = -g(t)$ \citep{dahl2021primal}.
It can also be characterized as the solution to the following strictly convex problem:
\begin{equation}
\textstyle\argmin_{s \in \intr(\K)} \bigl( f(s) + \tfrac{1}{2} \lVert s \rVert^2 \bigr).
\label{eq:centralpoint}
\end{equation}
For the nonnegative cone $\K = \bbR_{\geq}$, $f(s) = -\log(s)$ is an LHSCB with $\nu = 1$, and we have $g(s) = -s^{-1}$ and the central point $t = 1 = -g(1)$.
For some of Hypatia's cones, we are not aware of a simple analytic expression for the central point, in which case we typically use a non-central interior point.

%% file: form.tex
\section{General conic form and certificates}
\label{sec:form}

In \cref{sec:form:form,sec:form:cert}, we describe our general conic primal-dual form and the associated conic certificates.
In \cref{sec:form:hsde}, we introduce the homogeneous self-dual embedding (HSDE) conic feasibility problem, a solution of which may provide a conic certificate.

\subsection{General conic form}
\label{sec:form:form}

Hypatia uses the following primal conic form over variable $x \in \bbR^n$:
\begin{subequations}
\begin{align}
\textstyle\inf_x \quad c' x &:
\label{eq:prim:obj}
\\
b - A x &= 0,
\label{eq:prim:eq}
\\
h - G x &\in \K,
\label{eq:prim:K}
\end{align}
\label{eq:prim}
\end{subequations}
where $c \in \bbR^n$, $b \in \bbR^p$, and $h \in \bbR^q$ are vectors, $A : \bbR^n \to \bbR^p$ and $G : \bbR^n \to \bbR^q$ are linear maps, and $\K \subset \bbR^q$ is a Cartesian product $\K = \K_1 \times \cdots \times \K_K$ of exotic cones.
For $k \in \iin{K}$, we let $q_k = \dim(\K_k)$, so $\sum_k q_k = q = \dim(\K)$.
Henceforth we use $n, p, q$ to denote respectively the variable, equality, and conic constraint dimensions of a conic problem.

Once a proper cone $\K_k$ is defined through Hypatia's generic cone interface, both $\K_k$ and $\K_k^\ast$ may be used in any combination with other cones recognized by Hypatia to construct the Cartesian product cone $\K$ in \cref{eq:prim:K}.
The primal form \cref{eq:prim} matches CVXOPT's form, however CVXOPT only recognizes symmetric cones \citep{vandenberghe2010cvxopt}.
Unlike the conic form used by \citet{skajaa2015homogeneous,papp2021alfonso}, which recognizes conic constraints of the form $x \in \K$, our form does not require introducing slack variables to represent a more general constraint $h - G x \in \K$.

The conic dual problem of \cref{eq:prim}, over variables $y \in \bbR^p$ and $z \in \bbR^q$ associated with \cref{eq:prim:eq,eq:prim:K}, is:
\begin{subequations}
\begin{align}
\textstyle\sup_{y, z} \quad -b' y - h' z &:
\label{eq:dual:obj}
\\
c + A' y + G' z &= 0,
\label{eq:dual:eq}
\\
z &\in \K^\ast,
\label{eq:dual:K}
\end{align}
\label{eq:dual}
\end{subequations}
where \cref{eq:dual:eq} is associated with the primal variable $x \in \bbR^n$.

\subsection{Conic certificates}
\label{sec:form:cert}

Under certain conditions, there exists a simple conic certificate providing an easily verifiable proof of infeasibility of the primal \cref{eq:prim} or dual \cref{eq:dual} problem (via the conic generalization of Farkas' lemma) or optimality of a given primal-dual solution.
\begin{description}
\item[A primal improving ray] $x$ is a feasible direction for the primal along which the objective improves:
\begin{subequations}
\begin{align}
c' x &< 0,
\label{eq:dualinf:obj}
\\
-A x &= 0,
\label{eq:dualinf:eq}
\\
-G x &\in \K,
\label{eq:dualinf:K}
\end{align}
\label{eq:dualinf}
\end{subequations}
and hence it certifies dual infeasibility.
\item[A dual improving ray] $(y, z)$ is a feasible direction for the dual along which the objective improves:
\begin{subequations}
\begin{align}
-b' y - h' z &> 0,
\label{eq:priminf:obj}
\\
A' y + G' z &= 0,
\label{eq:priminf:eq}
\\
z &\in \K^\ast,
\label{eq:priminf:K}
\end{align}
\label{eq:priminf}
\end{subequations}
and hence it certifies primal infeasibility.
\item[A complementary solution] $(x, y, z)$ satisfies the primal-dual feasibility conditions \cref{eq:prim:eq,eq:prim:K,eq:dual:eq,eq:dual:K}, and has equal and attained primal and dual objective values:
\begin{equation}
c' x = -b' y - h' z,
\label{eq:compl}
\end{equation}
and hence certifies optimality of $(x, y, z)$ via conic weak duality.
\end{description}
One of these certificates exists if neither the primal nor the dual is \emph{ill-posed}.
Intuitively, according to \citet[Section 7.2]{mosek2020modeling}, a conic problem is ill-posed if a small perturbation of the problem data can change the feasibility status of the problem or cause arbitrarily large perturbations to the optimal solution (see \citet{permenter2017solving} for more details).

\subsection{Homogeneous self-dual embedding}
\label{sec:form:hsde}

The HSDE is a self-dual conic feasibility problem in variables $x \in \bbR^n, y \in \bbR^p, z \in \bbR^q, \tau \in \bbR, s \in \bbR^q, \kappa \in \bbR$ (see \citet[Section 6]{vandenberghe2010cvxopt}), derived from a homogenization of the primal-dual optimality conditions \cref{eq:prim:eq,eq:prim:K,eq:dual:eq,eq:dual:K,eq:compl}:
\begin{subequations}
\begin{align}
\begin{bmatrix}
0 \\ 0 \\ s \\ \kappa
\end{bmatrix}
& =
\begin{bmatrix}
0 & A' & G' & c \\
-A & 0 & 0 & b \\
-G & 0 & 0 & h \\
-c' & -b' & -h' & 0
\end{bmatrix}
\begin{bmatrix}
x \\ y \\ z \\ \tau
\end{bmatrix},
\label{eq:hsde:eq}
\\
(z, \tau, s, \kappa)
& \in \bigl( 
\K^\ast \times \bbR_{\geq} \times \K \times \bbR_{\geq} 
\bigr).
\label{eq:hsde:K}
\end{align}
\label{eq:hsde}
\end{subequations}
For convenience we let $\omega \coloneqq (x, y, z, \tau, s, \kappa) \in \bbR^{n + p + 2q + 2}$ represent a point.
We define the structured $4 \times 6$ block matrix $E \in \bbR^{(n + p + q + 1) \times \dim(\omega)}$ such that \cref{eq:hsde:eq} is equivalent to:
\begin{equation}
E \omega = 0.
\end{equation}
Here we assume $E$ has full row rank; in \cref{sec:linalg} we discuss preprocessing techniques that handle linearly dependent rows.
Note that $\omega = 0$ satisfies \cref{eq:hsde}, so the HSDE is always feasible.
A point $\omega$ is an interior point if it is strictly feasible for the conic constraints \cref{eq:hsde:K}, i.e.\ $\omega$ satisfies $(z, \tau, s, \kappa) \in \intr \bigl( \K^\ast \times \bbR_{\geq} \times \K \times \bbR_{\geq} \bigr)$.

Suppose a point $\omega$ is feasible for the HSDE \cref{eq:hsde}.
From skew symmetry of the square $4 \times 4$ block matrix in \cref{eq:hsde:eq}, we have $s' z + \kappa \tau = 0$.
From the conic constraints \cref{eq:hsde:K} and the dual cone inequality \cref{eq:Kdual} we have $s' z \geq 0$ and $\kappa \tau \geq 0$.
Hence $s' z = \kappa \tau = 0$.
We consider an exhaustive list of cases below.
\begin{description}
\item[Optimality.] 
If $\tau > 0, \kappa = 0$, then $(x, y, z) / \tau$ is a complementary solution satisfying the primal-dual optimality conditions \cref{eq:prim:eq,eq:prim:K,eq:dual:eq,eq:dual:K,eq:compl}.
\item[Infeasibility.] 
If $\tau = 0, \kappa > 0$, then $c' x + b' y + h' z < 0$ and we consider two sub-cases.
\begin{description}
\item[Of primal.] 
If $b' y + h' z < 0$, then $(y, z)$ is a primal infeasibility certificate satisfying \cref{eq:priminf}.
\item[Of dual.] 
If $c' x < 0$, then $x$ is a dual infeasibility certificate satisfying \cref{eq:dualinf}.
\end{description}
\item[No information.] 
If $\tau = \kappa = 0$, then $\omega$ provides no information about the feasibility or optimal values of the primal or dual.
\end{description}
Thus an HSDE solution $\omega$ satisfying $\kappa + \tau > 0$ provides an optimality or infeasibility certificate (see \citet[Lemma 1]{skajaa2015homogeneous} and \citet[Section 6.1]{vandenberghe2010cvxopt}).

According to \citet[Section 2]{skajaa2015homogeneous}, if the primal and dual problems are both feasible and have zero duality gap, SY (their algorithm) finds an HSDE solution with $\tau > 0$ (yielding a complementary solution), and if the primal or dual (possibly both) is infeasible, SY finds an HSDE solution with $\kappa > 0$ (yielding an infeasibility certificate).
This implies that if SY finds a solution with $\kappa = \tau = 0$, then $\kappa = \tau = 0$ for all solutions to the HSDE; in this case, no complementary solution or improving ray exists, and the primal or dual (possibly both) is ill-posed \citep{permenter2017solving}.
The algorithm we describe in \cref{sec:alg} is an extension of SY that inherits these properties.

%% file: alg.tex
\section{Central path following algorithm}
\label{sec:alg}

In \cref{sec:alg:cp}, we describe the central path of the HSDE, and in \cref{sec:alg:prox} we define central path proximity measures.
In \cref{sec:alg:alg}, we outline a high level PDIPM that maintains iterates close to the central path, and we give numerical convergence criteria for detecting approximate conic certificates.
In \cref{sec:alg:dir}, we derive prediction and centering directions and our corresponding TOA directions using the TOO.
Finally in \cref{sec:alg:step}, we summarize an SY-like stepping procedure and describe our sequence of four enhancements to this procedure.

\subsection{Central path of the HSDE}
\label{sec:alg:cp}

We define the HSDE in \cref{eq:hsde}.
Recall that $\K$ in our primal conic form \cref{eq:prim} is a Cartesian product $\K = \K_1 \times \cdots \times \K_K$ of $K$ exotic cones.
We partition the exotic cone indices $\iin{K}$ into two sets: $\Kpr$ for cones with primal oracles (i.e.\ for $\K_k$) and $\Kdu$ for cones with dual oracles (i.e.\ for $\K_k^\ast$).
For convenience, we append the $\tau$ and $\kappa$ variables onto the $s$ and $z$ variables.
Letting $\bKn = K + 1$, we define for $k \in \iin{\bKn}$:
\begin{subequations}
\begin{align}
\bK_k & \coloneqq \begin{cases}
\K_k & k \in \Kpr, \\
\K_k^\ast & k \in \Kdu, \\
\bbR_{\geq} & k = \bKn,
\end{cases}
\\
(\bz_k, \bs_k) & \coloneqq \begin{cases}
(z_k, s_k) & k \in \Kpr, \\
(s_k, z_k) & k \in \Kdu, \\
(\kappa, \tau) & k = \bKn.
\end{cases}
\end{align}
\end{subequations}

For a given initial interior point $\omega^0 = (x^0, y^0, z^0, \tau^0, s^0, \kappa^0)$, the central path of the HSDE is the trajectory of solutions $\omega_\mu = (x_\mu, y_\mu, z_\mu, \tau_\mu, s_\mu, \kappa_\mu)$, parameterized by $\mu > 0$, satisfying:
\begin{subequations}
\begin{align}
E \omega_\mu &= \mu E \omega^0,
\label{eq:cp:eq}
\\
\bz_{\mu,k} + \mu g_k(\bs_{\mu,k})
&= 0 
\quad \forall k \in \iin{\bKn},
\label{eq:cp:psi}
\\
(\bz_\mu, \bs_\mu) &\in \intr (\bK^\ast \times \bK).
\label{eq:cp:K}
\end{align}
\label{eq:cp}
\end{subequations}
When all exotic cones have primal oracles (i.e.\ $\Kdu$ is empty), our definition \cref{eq:cp} exactly matches the central path defined in \citet[Equation 32]{vandenberghe2010cvxopt}, and only differs from the definition in \citet[Equations 7-8]{skajaa2015homogeneous} in the affine form (i.e.\ the variable names and affine constraint structure).
Unlike SY, our central path condition \cref{eq:cp:psi} allows cones with dual oracles ($\Kdu$ may be nonempty).

To obtain an initial point $\omega^0$, we first let:
\begin{equation}
\bigl( \bz^0_k, \bs^0_k \bigr) =
( -g_k (t_k), t_k )
\quad \forall k \in \iin{\bKn},
\label{eq:initsz}
\end{equation}
where $t_k \in \intr \bigl( \bK_k \bigr)$ is the initial interior point oracle (note that $\tau^0 = \kappa^0 = 1$).
Although $x^0$ and $y^0$ can be chosen arbitrarily, we let $x^0$ be the solution of:
\begin{subequations}
\begin{align}
\textstyle\min_{x \in \bbR^n} \quad 
\lVert x \rVert & :
\\
-A x + b \tau^0 &= 0,
\\
-G x + h \tau^0 - s^0 &= 0,
\label{eq:init:x}
\end{align}
\end{subequations}
and we let $y^0$ be the solution of:
\begin{subequations}
\begin{align}
\textstyle\min_{y \in \bbR^p} \quad 
\lVert y \rVert & :
\\
A' y + G' z^0 + c \tau^0 &= 0.
\label{eq:init:y}
\end{align}
\end{subequations}
In \cref{sec:linalg}, we outline a QR-factorization-based procedure for preprocessing the affine data of the conic model and solving for $\omega^0$. 

Like \citet[Section 4.1]{skajaa2015homogeneous}, we define the \emph{complementarity gap} function:
\begin{equation}
\mu(\omega) \coloneqq 
\bs' \bz / \tsum{k \in \iin{\bKn}} \nu_k,
\label{eq:mu}
\end{equation}
where $\nu_k$ is the LHSCB parameter of the LHSCB $f_k$ for $\bK_k$ (see \cref{eq:lhscb:lh}).
Note that $\mu(\omega) > 0$ if $(\bz, \bs) \in \bK^\ast \times \bK$, by a strict version of the dual cone inequality \cref{eq:Kdual}.
From \cref{eq:initsz}, $\mu(\omega^0) = 1$, since in \cref{eq:mu} we have $(\bs^0)' \bz^0 = \tsum{k \in \iin{\bKn}} t_k' (-g_k (t_k))$, and $t_k' (-g_k (t_k)) = \nu_k$ by logarithmic homogeneity of $f_k$ \citep[Proposition 2.3.4]{nesterov1994interior}.
Hence $\omega^0$ satisfies the central path conditions \cref{eq:cp} for parameter value $\mu = 1$.
The central path is therefore a trajectory that starts at $\omega^0$ with complementarity gap $\mu = 1$ and approaches a solution for the HSDE as $\mu$ decreases to zero.

\subsection{Central path proximity}
\label{sec:alg:prox}

Given a point $\omega$, we define the central path \emph{proximity} $\pi_k$ for exotic cone $k \in \iin{\bKn}$ as:
\begin{equation}
\pi_k(\omega) \coloneqq \begin{cases}
\big\lVert (H_k (\bs_k))^{-1 / 2} ( \bz_k / \mu(\omega) + g_k(\bs_k) ) \big\rVert
& \text{if } \mu(\omega) > 0, \bs_k \in \intr \bigl( \bK_k \bigr),
\\
\infty & \text{otherwise}.
\end{cases}
\label{eq:proxhessk}
\end{equation}
Hence $\pi_k$ is a measure of the distance from $\bs_k$ and $\bz_k$ to the surface defined by the central path condition \cref{eq:cp:psi} (compare to \citet[Equation 9]{skajaa2015homogeneous} and \citet[Section 4]{nesterov1998primal}).

In \cref{lem:interior}, we show that for exotic cone $k \in \iin{\bKn}$, if $\pi_k(\omega) < 1$, then $\bs_k \in \intr \bigl( \bK_k \bigr)$ and $\bz_k \in \intr \bigl( \bK^\ast_k \bigr)$.
This condition is sufficient but not necessary for strict cone feasibility.
If it holds for all $k \in \iin{\bKn}$, then $\omega$ is an interior point (by definition) and \cref{eq:cp:K} is satisfied.
From \cref{eq:proxhessk}, $\pi_k(\omega)$ can be computed by evaluating the feasibility check, gradient, and Hessian oracles for $\bK_k$ at $\bs_k$.

\begin{lemma}
Given a point $\omega$, for each $k \in \iin{\bKn}$, $\pi_k(\omega) < 1$ implies $\bs_k \in \intr \bigl( \bK_k \bigr)$ and $\bz_k \in \intr \bigl( \bK^\ast_k \bigr)$.
\label{lem:interior}
\end{lemma}

\begin{proof}
We adapt \citet[Lemma 15]{papp2017homogeneous}.
Fix $\mu = \mu(\omega)$ for convenience, and suppose $\pi_k(\omega) < 1$ for exotic cone $k \in \iin{\bKn}$.
Then by \cref{eq:proxhessk}, $\mu > 0$ and $\bs_k \in \intr \bigl( \bK_k \bigr)$.
By \citet[Theorem 8]{papp2017homogeneous}, $\bs_k \in \intr \bigl( \bK_k \bigr)$ implies $-g_k (\bs_k) \in \intr \bigl( \bK^\ast_k \bigr)$.
Let $f_k$ be the LHSCB for $\bK_k$, and let $H^\ast_k \coloneqq \nabla^2 f^\ast_k$ denote the Hessian operator for the conjugate $f^\ast_k$ (see \cref{eq:conj}) of $f_k$.
By \citet[Equation 13]{papp2017homogeneous}, $H^\ast_k (-g_k (\bs_k)) = (H_k (\bs_k))^{-1}$, so:
\begin{subequations}
\begin{align}
& \big\lVert ( H^\ast_k (-g_k (\bs_k)) )^{1 / 2} ( \bz_k / \mu + g_k (\bs_k) ) \big\rVert
\\
&= \big\lVert (H_k (\bs_k))^{-1 / 2} ( \bz_k / \mu + g_k (\bs_k) ) \big\rVert
\\
&= \pi_k (\omega) < 1.
\label{eq:interior}
\end{align}
\end{subequations}
So by \citet[Definition 1]{papp2017homogeneous}, $\bz_k / \mu \in \intr \bigl( \bK^\ast_k \bigr)$, hence $\bz_k \in \intr \bigl( \bK^\ast_k \bigr)$.
\end{proof}

We now define a proximity function that aggregates the exotic cone central path proximity values $\pi_k(\omega) \geq 0, \forall k \in \iin{\bKn}$.
SY aggregates by taking the $\ell_2$ norm:
\begin{equation}
\pitwo (\omega) \coloneqq 
\big\lVert ( \pi_k(\omega) )_{k \in \iin{\bKn}} \big\rVert.
\label{eq:proxagg:l2}
\end{equation}
An alternative aggregated proximity uses the $\ell_{\infty}$ norm (maximum):
\begin{equation}
\piinf (\omega) \coloneqq 
\big\lVert ( \pi_k(\omega) )_{k \in \iin{\bKn}} \big\rVert_{\infty}.
\label{eq:proxagg:linf}
\end{equation}
Clearly, $0 \leq \pi_k (\omega) \leq \piinf (\omega) \leq \pitwo (\omega), \forall k \in \iin{\bKn}$.
Both conditions $\pitwo (\omega) < 1$ and $\piinf (\omega) < 1$ guarantee by \cref{lem:interior} that $\omega$ is an interior point, however using $\pitwo$ leads to a more restrictive condition on $\omega$.

\subsection{High level algorithm}
\label{sec:alg:alg}

We describe a high level algorithm for approximately solving the HSDE.
The method starts at the initial interior point $\omega^0$ with complementarity gap $\mu(\omega^0) = 1$ and approximately tracks the central path trajectory \cref{eq:cp} through a series of iterations.
It maintains feasibility for the linear equality conditions \cref{eq:cp:eq} and strict cone feasibility conditions \cref{eq:cp:K}, but allows violation of the nonlinear equality conditions \cref{eq:cp:psi}.
On the $i$th iteration, the current interior point is $\omega^{i-1}$ satisfying $\pi_k(\omega^{i-1}) < 1, \forall k \in \iin{\bKn}$, and the complementarity gap is $\mu(\omega^{i-1})$.
The method searches for a new point $\omega^i$ that maintains the proximity condition $\pi_k(\omega^i) < 1, \forall k \in \iin{\bKn}$ (and hence is an interior point) and either has a smaller complementarity gap $\mu(\omega^i) < \mu(\omega^{i-1})$ or a smaller aggregate proximity value $\pi(\omega^i) < \pi(\omega^{i-1})$ (where $\pi$ is $\pitwo$ or $\piinf$), or both.
As the complementarity gap decreases towards zero, the RHS of \cref{eq:cp:eq} approaches the origin, so the iterates approach a solution of the HSDE \cref{eq:hsde}.

To detect an approximate conic certificate and terminate the iterations, we check if the current iterate $\omega$ satisfies any of the following numerical convergence criteria.
These conditions use positive tolerance values for feasibility $\varepsilon_f$, infeasibility $\varepsilon_i$, absolute gap $\varepsilon_a$, relative gap $\varepsilon_r$, and ill-posedness $\varepsilon_p$ (see \cref{sec:testing:method} for the tolerance values we use in computational testing).
\begin{description}
\item[Optimality.]
\label{eq:optimality}
We terminate with a complementary solution $(x, y, z) / \tau$ approximately satisfying the primal-dual optimality conditions \cref{eq:prim:eq,eq:prim:K,eq:dual:eq,eq:dual:K,eq:compl} if:
\begin{subequations}
\begin{equation}
\max \biggl( 
\frac{ \lVert A' y + G' z + c \tau \rVert_{\infty} }{ 1 + \lVert c \rVert_{\infty} }, 
\frac{ \lVert -A x + b \tau \rVert_{\infty} }{ 1 + \lVert b \rVert_{\infty} }, 
\frac{ \lVert -G x + h \tau - s \rVert_{\infty} }{ 1 + \lVert h \rVert_{\infty} } 
\biggr) \leq \varepsilon_f \tau,
\label{eq:term:opt:res}
\end{equation}
and at least one of the following two conditions holds: 
\begin{align}
s' z &\leq \varepsilon_a,
\label{eq:term:opt:absgap}
\\
\min ( 
s' z / \tau, \lvert c' x + b' y + h' z \rvert 
) 
&\leq \varepsilon_r \max ( 
\tau, \min ( \lvert c' x \rvert, \lvert b' y + h' z \rvert ) 
).
\label{eq:term:opt:relgap}
\end{align}
\label{eq:term:opt}
\end{subequations}
Note that \cref{eq:term:opt:absgap,eq:term:opt:relgap} are absolute and relative optimality gap conditions respectively.
\item[Primal infeasibility.]
We terminate with a dual improving ray $(y, z)$ approximately satisfying \cref{eq:priminf} if: 
\begin{equation}
b' y + h' z < 0,
\qquad
\lVert A' y + G' z \rVert_{\infty} 
\leq -\varepsilon_i ( b' y + h' z ).
\label{eq:term:priminf}
\end{equation}
\item[Dual infeasibility.]
We terminate with a primal improving ray $x$ approximately satisfying \cref{eq:dualinf} if: 
\begin{equation}
c' x < 0, 
\qquad 
\max ( 
\lVert A x \rVert_{\infty}, 
\lVert G x + s \rVert_{\infty} 
) \leq -\varepsilon_i c' x.
\label{eq:term:dualinf}
\end{equation}
\item[Ill-posed primal or dual.]
If $\tau$ and $\kappa$ are approximately $0$, the primal and dual problem statuses cannot be determined (see \cref{sec:form:hsde}).
We terminate with an ill-posed status if:
\begin{equation}
\mu(\omega) \leq \varepsilon_p, \qquad 
\tau \leq \varepsilon_p \min (1, \kappa).
\label{eq:term:ill}
\end{equation}
\end{description}

The high level path following algorithm below computes an approximate solution to the HSDE.
In \cref{sec:alg:step}, we describe specific stepping procedures for \cref{line:step}.
\begin{algorithmic}[1]
\Procedure {SolveHSDE}{}
\State compute initial interior point $\omega^0$
\State $i \gets 1$
\While {$\omega^{i-1}$ does not satisfy any of the convergence conditions \crefrange{eq:term:opt}{eq:term:ill}} 
\State $\omega^i \gets $ \Call{Step}{$\omega^{i-1}$}
\label{line:step}
\State $i \gets i + 1$
\EndWhile
\State \Return $\omega^i$
\EndProcedure
\label{proc:solve}
\end{algorithmic}

\subsection{Search directions}
\label{sec:alg:dir}

At a given iteration of the path following method, let $\omega$ be the current interior point and fix $\mu = \mu(\omega)$ for convenience.
The stepping procedures we describe in \cref{sec:alg:step} first compute one or more search directions, which depend on $\omega$.
We derive the \emph{centering} direction in \cref{sec:alg:dir:cent} and the \emph{prediction} direction in \cref{sec:alg:dir:pred}.
The goal of centering is to step to a point with a smaller aggregate central path proximity than the current point, i.e.\ to step towards the central path.
The goal of prediction is to step to a point with a smaller complementarity gap, i.e.\ to step closer to a solution of the HSDE.
The centering and prediction directions match those used by SY.
We associate with each of these directions a new \emph{third order adjustment} (TOA) direction, which depends on the TOO and helps to correct the corresponding unadjusted direction (which must be computed before the TOA direction).
Hence we derive four types of directions here.

Each direction is computed as the solution to a linear system with a structured square $6 \times 6$ block matrix left hand side (LHS) and a particular right hand side (RHS) vector.
The LHS, which depends only on $\omega$ and the problem data, is the same for all four directions at a given iteration.
We let $r \coloneqq (r_E, r_1, \ldots, r_{\bKn}) \in \bbR^{\dim(\omega)}$ represent an RHS, where $r_E \in \bbR^{n + p + q + 1}$ corresponds to the linear equalities \cref{eq:cp:eq} and $r_k \in \bbR^{q_k}, \forall k \in \iin{\bKn}$ corresponds to the nonlinear equalities \cref{eq:cp:psi}.
The direction $\delta \coloneqq (\delta_x, \delta_y, \delta_z, \delta_\tau, \delta_s, \delta_\kappa) \in \bbR^{\dim(\omega)}$ corresponding to $r$ is the solution to:
\begin{subequations}
\begin{align}
E \delta &= r_E,
\label{eq:dirs:eq}
\\
\delta_{\bz,k} + \mu H_k (\bs_k) \delta_{\bs,k} 
&= r_k
\quad \forall k \in \iin{\bKn}.
\label{eq:dirs:K}
\end{align}
\label{eq:dirs}
\end{subequations}
Since $E$ is assumed to have full row rank and each $H_k$ is positive definite, this square system is nonsingular and hence has a unique solution.
In \cref{sec:linalg}, we describe a particular method for solving \cref{eq:dirs}.

\subsubsection{Centering}
\label{sec:alg:dir:cent}

The centering direction $\delta^c$ is analogous to the definition of \citet[Section 3.2]{skajaa2015homogeneous}.
It reduces the violation on the central path nonlinear equality condition \cref{eq:cp:psi} (and can be interpreted as a Newton step), while keeping the complementarity gap $\mu$ (approximately) constant.
We denote the centering TOA direction $\delta^{ct}$.
To maintain feasibility for the linear equality condition \cref{eq:cp:eq}, we ensure $E \delta^c = E \delta^{ct} = 0$ in \cref{eq:dirs:eq}.

Dropping the index $k \in \iin{\bKn}$ for conciseness, recall that \cref{eq:cp:psi} expresses $\bz + \mu g(\bs) = 0$.
A first order approximation of this condition gives:
\begin{subequations}
\begin{align}
\bz + \delta_{\bz} + 
\mu ( g(\bs) + H(\bs) \delta_{\bs} )
&= 0
\\
\Rightarrow \quad 
\delta_{\bz} + \mu H(\bs) \delta_{\bs} 
&= -\bz - \mu g(\bs),
\label{eq:cent1}
\end{align}
\end{subequations}
which matches the form of \cref{eq:dirs:K}.
Hence we let the centering direction $\delta^c$ be the solution to:
\begin{subequations}
\begin{align}
E \delta &= 0,
\\
\delta_{\bz,k} + \mu H_k (\bs_k) \delta_{\bs,k} 
&= -\bz_k - \mu g_k (\bs_k) 
\quad \forall k \in \iin{\bKn}.
\label{eq:centsys:K}
\end{align}
\label{eq:centsys}
\end{subequations}
Similarly, a second order approximation of $\bz + \mu g(\bs) = 0$ gives:
\begin{subequations}
\begin{align}
\bz + \delta_{\bz} + \mu \bigl( 
g(\bs) + H(\bs) \delta_{\bs} + 
\tfrac{1}{2} \nabla^3 f(\bs) [\delta_{\bs}, \delta_{\bs}] 
\bigr)
&= 0
\\
\Rightarrow \quad
\delta_{\bz} + \mu H(\bs) \delta_{\bs} 
&= -\bz - \mu g(\bs) + \mu \Tau (\bs, \delta_{\bs}),
\label{eq:cent2}
\end{align}
\end{subequations}
where \cref{eq:cent2} uses the definition of the TOO in \cref{eq:too}.
Note that the RHSs of \cref{eq:cent1,eq:cent2} differ only by $\mu \Tau (\bs, \delta_{\bs})$, which depends on $\delta_{\bs}$.
To remove this dependency, we substitute the centering direction $\delta^c$, which we assume is already computed, into the RHS of \cref{eq:cent2}.
Hence we let the centering TOA direction $\delta^{ct}$, which adjusts the centering direction, be the solution to:
\begin{subequations}
\begin{align}
E \delta &= 0,
\\
\delta_{\bz,k} + \mu H_k (\bs_k) \delta_{\bs,k} 
&= \mu \Tau_k \bigl( \bs_k, \delta^c_{\bs,k} \bigr)
\quad \forall k \in \iin{\bKn}.
\label{eq:centtoasys:K}
\end{align}
\label{eq:centtoasys}
\end{subequations}
We note that for a rescaling factor $\alpha \in (0,1)$, the TOA direction corresponding to $\alpha \delta^c$ (a rescaling of the centering direction) is $\alpha^2 \delta^{ct}$ (a rescaling of the centering TOA direction).

\subsubsection{Prediction}
\label{sec:alg:dir:pred}

The prediction direction $\delta^p$ reduces the complementarity gap and is analogous to the definition of \citet[Section 3.1]{skajaa2015homogeneous}.
We derive $\delta^p$ and its corresponding TOA direction $\delta^{pt}$ by considering the central path conditions \cref{eq:cp} as a dynamical system parametrized by $\mu > 0$, and differentiating the linear and nonlinear equalities \cref{eq:cp:eq,eq:cp:psi}.

Differentiating \cref{eq:cp:eq} once gives:
\begin{equation}
E \dot{\omega}_\mu = E \omega^0.
\label{eq:predE1}
\end{equation}
Rescaling \cref{eq:predE1} by $-\mu$ and substituting \cref{eq:cp:eq} gives:
\begin{equation}
E (-\mu \dot{\omega}_\mu) 
= -\mu E \omega^0
= -E \omega_\mu.
\label{eq:predE1b}
\end{equation}
Dropping the index $k \in \iin{\bKn}$ for conciseness, we differentiate $\bz_{\mu} + \mu g(\bs_{\mu}) = 0$ from \cref{eq:cp:psi} once to get:
\begin{equation}
\dot{\bz}_{\mu} + g(\bs_{\mu}) + \mu H(\bs_{\mu}) \dot{\bs}_{\mu} 
= 0.
\label{eq:predK1}
\end{equation}
Rescaling \cref{eq:predK1} by $-\mu$ and substituting $\bz_{\mu} = -\mu g(\bs_{\mu})$ from \cref{eq:cp:psi} gives:
\begin{equation}
{-\mu} \dot{\bz}_{\mu} + \mu H(\bs_{\mu}) (-\mu \dot{\bs}_{\mu}) 
= -\bz_\mu.
\label{eq:predK1b}
\end{equation}
The direction $\dot{\omega}_\mu$ is tangent to the central path.
Like SY, we interpret the prediction direction as $\delta^p = -\mu \dot{\omega}_{\mu}$, so \cref{eq:predE1b,eq:predK1b} become:
\begin{subequations}
\begin{align}
E \delta^p
&= -E \omega_\mu,
\\
\delta^p_{\bz} + \mu H(\bs_{\mu}) \delta^p_{\bs} 
&= -\bz_\mu,
\end{align}
\label{eq:predall}
\end{subequations}
which matches the form \cref{eq:dirs}.
So we let $\delta^p$ be the solution to:
\begin{subequations}
\begin{align}
E \delta &= -E \omega,
\\
\delta_{\bz,k} + \mu H_k (\bs_k) \delta_{\bs,k} 
&= -\bz_k
\quad \forall k \in \iin{\bKn}.
\label{eq:predsys:K}
\end{align}
\label{eq:predsys}
\end{subequations}
Differentiating \cref{eq:cp:eq} twice and rescaling by $\frac{1}{2} \mu^2$ gives:
\begin{equation}
E \bigl( \tfrac{1}{2} \mu^2 \ddot{\omega}_\mu \bigr) 
= 0.
\label{eq:predE2}
\end{equation}
Differentiating $\bz_{\mu} + \mu g(\bs_{\mu}) = 0$ twice gives:
\begin{equation}
\ddot{\bz}_{\mu} + 2 H(\bs_{\mu}) \dot{\bs}_{\mu} + 
\mu \nabla^3 f(\bs_{\mu}) [\dot{\bs}_{\mu}, \dot{\bs}_{\mu}] + 
\mu H(\bs_{\mu}) \ddot{\bs}_{\mu} 
= 0.
\label{eq:predK2}
\end{equation}
Rescaling \cref{eq:predK2} by $\frac{1}{2} \mu^2$ and substituting the TOO definition \cref{eq:too}, we have:
\begin{subequations}
\begin{align}
\tfrac{1}{2} \mu^2 \ddot{\bz}_{\mu} + 
\mu H(\bs_{\mu}) \bigl( \tfrac{1}{2} \mu^2 \ddot{\bs}_{\mu} \bigr) 
&= \mu H(\bs_{\mu}) (-\mu \dot{\bs}_{\mu}) - 
\tfrac{1}{2} \mu \nabla^3 f(\bs_{\mu}) [-\mu \dot{\bs}_{\mu}, -\mu \dot{\bs}_{\mu}]
\\
&= \mu H(\bs_{\mu}) (-\mu \dot{\bs}_{\mu}) + 
\mu \Tau(\bs_{\mu}, -\mu \dot{\bs}_{\mu}).
\label{eq:predK2b}
\end{align}
\end{subequations}
We interpret the prediction TOA direction, which adjusts the prediction direction, as $\delta^{pt} = \frac{1}{2} \mu^2 \ddot{\omega}$.
The RHS of \cref{eq:predK2b} depends on $\dot{\bs}_{\mu}$, so we remove this dependency by substituting the prediction direction $\delta^p = -\mu \dot{\omega}_{\mu}$, which we assume is already computed.
Hence using \cref{eq:predE2,eq:predK2b}, we let $\delta^{pt}$ be the solution to:
\begin{subequations}
\begin{align}
E \delta &= 0,
\\
\delta_{\bz,k} + \mu H_k (\bs_k) \delta_{\bs,k} 
&= \mu H_k(\bs_k) \delta^p_{\bs,k} + 
\mu \Tau_k \bigl( \bs_k, \delta^p_{\bs,k} \bigr) 
\quad \forall k \in \iin{\bKn}.
\label{eq:predtoasys:K}
\end{align}
\label{eq:predtoasys}
\end{subequations}
We note that the RHS in \cref{eq:predtoasys:K} differs from the `higher order corrector' RHS proposed by \citet[Equation 16]{dahl2021primal}, which has the form $\frac{1}{2} \nabla^3 f_k \bigl[ \delta_{\bs,k}^p, (H_k(\bs_k))^{-1} \delta_{\bz,k}^p \bigr]$.
For example, our form does not satisfy all of the properties in \citet[Lemmas 3 and 4]{dahl2021primal}.

\subsection{Stepping procedures}
\label{sec:alg:step}

A stepping procedure computes one or more directions from \cref{sec:alg:dir} and uses the directions to search for a new interior point.
Recall from \cref{line:step} of the high level PDIPM in \cref{sec:alg:alg} that on iteration $i$ with current iterate $\omega^{i-1}$, \Call{Step}{} computes $\omega^i$ satisfying $\pi(\omega^i) < 1$ and either $\mu(\omega^i) < \mu(\omega^{i-1})$ (prediction) or $\pi(\omega^i) < \pi(\omega^{i-1})$ (centering) or both.
In \cref{sec:alg:step:basic}, we describe a stepping procedure similar to that of Alfonso \citep{papp2021alfonso}, which is a practical implementation of SY.
This procedure, which we call \emph{basic}, alternates between prediction and centering steps and does not use the TOA directions.
In \crefrange{sec:alg:step:prox}{sec:alg:step:comb}, we describe a sequence of four cumulative enhancements to the \emph{basic} procedure, with the goal of improving iteration counts and per-iteration computational efficiency in practice.
The main purpose of our computational testing in \cref{sec:testing} is to assess the value of these enhancements on a diverse set of benchmark instances.

\subsubsection{Basic stepping procedure}
\label{sec:alg:step:basic}

First, we decide whether to perform a centering step or a prediction step.
If the current iterate $\omega^{i-1}$ (at the $i$th iteration) is very close to the central path, i.e.\ if the sum proximity \cref{eq:proxagg:l2} does not exceed $\eta = 0.0332$ (from Alfonso \citep{papp2020alfonso}), or if the most recent $N = 4$ steps have all been centering steps, then we compute the prediction direction $\delta^p$ from \cref{eq:predsys}.
Otherwise, we compute the centering direction $\delta^c$ from \cref{eq:centsys}.
Letting $j$ be the number of consecutive centering steps taken immediately before the current $i$th iteration, the search direction is:
\begin{equation}
\delta \coloneqq \begin{cases}
\delta^p & \text{if $\pitwo(\omega^{i-1}) \leq \eta$ or $j \geq N$},
\\
\delta^c & \text{otherwise}.
\end{cases}
\label{eq:basic:delta}
\end{equation}

Next, we perform a backtracking line search in the direction $\delta$.
The search finds a step length $\hat{\alpha} \in (0, 1)$ from a fixed schedule of decreasing values $\mathcal{A} = \{ \alpha_l \}_{l \in \iin{L}}$, where $L = 18$, $\alpha_1 = 0.9999$, and $\alpha_L = 0.0005$.
The next iterate $\omega^i = \omega^{i-1} + \hat{\alpha} \delta$ becomes the first point in the backtracking line search that satisfies $\pitwo(\omega^i) \leq \beta_1$ for $\beta_1 = 0.2844$ (from Alfonso \citep{papp2020alfonso}), which guarantees interiority by \cref{lem:interior}.
If the backtracking search terminates without a step length satisfying the proximity condition (i.e.\ $\alpha_L$ is too large), the PDIPM algorithm terminates without a solution. 
In \cref{sec:search} we discuss our implementation of the proximity check that we run for each candidate point in the backtracking search.

The \emph{basic} stepping procedure is summarized as follows.
Note the centering step count $j$ is initialized to zero before the first iteration $i = 1$.
Since $\omega^0$ is exactly on the central path (i.e.\ the proximity is zero), the first iteration uses a prediction step.
\begin{algorithmic}[1]
\Procedure {BasicStep}{$\omega^{i-1}$, $j$}
\If {$\pitwo(\omega^{i-1}) \leq \eta$ or $j \geq N$}
\Comment{choose predict or center}
\State $\delta \gets \delta^p$ from \cref{eq:predsys}
\Comment{compute prediction direction}
\State $j \gets 0$
\Else
\State $\delta \gets \delta^c$ from \cref{eq:centsys}
\Comment{compute centering direction}
\State $j \gets j + 1$
\EndIf
\State $\hat{\alpha} \gets \max \{ \alpha \in \mathcal{A} : \pitwo ( \omega^{i-1} + \alpha \delta ) \leq \beta_1 \bigr\}$ 
\Comment{compute step length by backtracking search}
\State $\omega^i \gets \omega^{i-1} + \hat{\alpha} \delta$
\Comment{update current iterate}
\EndProcedure
\end{algorithmic}

\subsubsection{Less restrictive proximity}
\label{sec:alg:step:prox}

The \emph{basic} stepping procedure in \cref{sec:alg:step:basic} requires iterates to remain in close proximity to the central path and usually only takes prediction steps from iterates that are very close to the central path.
Although conservative proximity conditions are used to prove polynomial iteration complexity in \citet{papp2017homogeneous}, they may be too restrictive from the perspective of practical performance.
To allow prediction steps from a larger neighborhood of the central path, we use the $\piinf$ proximity measure from \cref{eq:proxagg:linf} instead of $\pitwo$ to compute the proximity of $\omega^{i-1}$, though we do not change the proximity bound $\eta$.
To allow longer step lengths, we also use $\piinf$ instead of $\pitwo$ for the backtracking search proximity checks, and we increase this proximity bound to $\beta_2 = 0.99$ (by \cref{lem:interior}, $\beta_2 < 1$ guarantees interiority).

The \emph{prox} stepping procedure, which enhances the \emph{basic} stepping procedure by relaxing the proximity conditions somewhat, is summarized as follows.
\begin{algorithmic}[1]
\Procedure {ProxStep}{$\omega^{i-1}$, $j$}
\If {$\piinf(\omega^{i-1}) \leq \eta$ or $j \geq N$}
\Comment{use less restrictive proximity measure $\piinf$}
\State $\delta \gets \delta^p$ from \cref{eq:predsys}
\State $j \gets 0$
\Else
\State $\delta \gets \delta^c$ from \cref{eq:centsys}
\State $j \gets j + 1$
\EndIf
\State $\hat{\alpha} \gets \max \{ \alpha \in \mathcal{A} : \piinf ( \omega^{i-1} + \alpha \delta ) \leq \beta_2 \}$
\Comment{use $\piinf$ and larger proximity bound $\beta_2$}
\State $\omega^i \gets \omega^{i-1} + \hat{\alpha} \delta$
\EndProcedure
\end{algorithmic}

\subsubsection{Third order adjustments}
\label{sec:alg:step:toa}

We modify the \emph{prox} stepping procedure in \cref{sec:alg:step:prox} to incorporate the new TOA directions associated with the prediction and centering directions.
After deciding whether to predict or center (using the same criteria as \emph{prox}), we compute the unadjusted direction $\delta^u$ (i.e.\ $\delta^p$ or $\delta^c$) and its associated TOA direction $\delta^t$ (i.e.\ $\delta^{pt}$ or $\delta^{ct}$).
We perform a backtracking line search in direction $\delta^u$, just like \emph{prox}, and we use this step length $\hat{\alpha}^u \in (0,1)$ to scale down the TOA direction.
We let the final direction be $\delta^u + \hat{\alpha}^u \delta^t$.
The rescaling of $\delta^t$ helps to prevent over-adjustment.
Finally, we perform a second backtracking line search, using the same techniques and proximity condition as the first line search.

The \emph{TOA} stepping procedure, which enhances the \emph{prox} stepping procedure by incorporating the TOA directions, is summarized as follows.
\begin{algorithmic}[1]
\Procedure {TOAStep}{$\omega^{i-1}$, $j$}
\If {$\piinf(\omega^{i-1}) \leq \eta$ or $j \geq N$}
\State $\delta^u \gets \delta^p$ from \cref{eq:predsys}
\State $\delta^t \gets \delta^{pt}$ from \cref{eq:predtoasys}
\Comment{compute prediction TOA direction}
\State $j \gets 0$
\Else
\State $\delta^u \gets \delta^c$ from \cref{eq:centsys}
\State $\delta^t \gets \delta^{ct}$ from \cref{eq:centtoasys}
\Comment{compute centering TOA direction}
\State $j \gets j + 1$
\EndIf
\State $\hat{\alpha}^u \gets \max \{ \alpha \in \mathcal{A} : \piinf ( \omega^{i-1} + \alpha \delta^u ) \leq \beta_2 \}$
\label{line:toa:alpha1}
\Comment{perform line search for unadjusted direction}
\State $\delta \gets \delta^u + \hat{\alpha}^u \delta^t$
\label{line:toa:delta}
\Comment{compute final direction}
\State $\hat{\alpha} \gets \max \{ \alpha \in \mathcal{A} : \piinf ( \omega^{i-1} + \alpha \delta ) \leq \beta_2 \}$
\label{line:toa:alpha}
\State $\omega^i \gets \omega^{i-1} + \hat{\alpha} \delta$
\EndProcedure
\end{algorithmic}

\subsubsection{Curve search}
\label{sec:alg:step:curve}

The \emph{TOA} stepping procedure in \cref{sec:alg:step:toa} performs two backtracking line searches, which can be quite expensive. 
We propose using a single backtracking search along a curve that is quadratic in the step parameter $\alpha$ and linear in the unadjusted and TOA directions.
Recall from \cref{line:toa:delta} of the \emph{TOA} procedure that we compute a direction $\delta$ as a linear function of the step parameter from the first line search. 
Substituting this $\delta$ function into the usual linear trajectory gives the curved trajectory $\omega^{i-1} + \alpha ( \delta^u + \alpha \delta^t )$ for $\alpha \in (0,1)$, where $\delta^u$ and $\delta^t$ are the unadjusted and TOA directions (as in the \emph{TOA} procedure).
Intuitively, a backtracking search along this curve achieves a more dynamic rescaling of the TOA direction.

The \emph{curve} stepping procedure, which enhances the \emph{TOA} stepping procedure by using a search on a curve instead of two line searches, is summarized as follows.
\begin{algorithmic}[1]
\Procedure {CurveStep}{$\omega^{i-1}$, $j$}
\If {$\piinf(\omega^{i-1}) \leq \eta$ or $j \geq N$}
\State $\delta^u \gets \delta^p$ from \cref{eq:predsys}
\State $\delta^t \gets \delta^{pt}$ from \cref{eq:predtoasys}
\State $j \gets 0$
\Else
\State $\delta^u \gets \delta^c$ from \cref{eq:centsys}
\State $\delta^t \gets \delta^{ct}$ from \cref{eq:centtoasys}
\State $j \gets j + 1$
\EndIf
\State let $\hat{\omega} (\alpha) 
\coloneqq \omega^{i-1} + \alpha ( \delta^u + \alpha \delta^t )$
\Comment{use curved trajectory}
\label{line:curve:traj}
\State $\hat{\alpha} \gets \max \{ \alpha \in \mathcal{A} : \piinf ( \hat{\omega} (\alpha) ) \leq \beta_2 \}$
\State $\omega^i \gets \hat{\omega} (\hat{\alpha})$
\label{line:curve:omega}
\EndProcedure
\end{algorithmic}

\subsubsection{Combined directions}
\label{sec:alg:step:comb}

Unlike \citet{skajaa2015homogeneous,papp2021alfonso}, most conic PDIPMs combine the prediction and centering phases (e.g.\ \citet{vandenberghe2010cvxopt,dahl2021primal}).
We propose using a single search on a curve that is quadratic in the step parameter $\alpha$ and linear in all four directions $\delta^c, \delta^{ct}, \delta^p, \delta^{pt}$ from \cref{sec:alg:step:toa}.
Intuitively, we can step further in a convex combination of the prediction and centering directions than we can in just the prediction direction.
In practice, a step length of one is usually ideal for the centering phase, so we can imagine performing a backtracking search from the point obtained from a pure prediction step (with step length one) towards the point obtained from a pure centering step, terminating when we are close enough to the centering point to satisfy the proximity condition.
This approach fundamentally differs from the previous procedures we have described because the search trajectory does not finish at the current iterate $\omega^{i-1}$.
If $\hat{\omega}^p (\alpha)$ and $\hat{\omega}^c (\alpha)$ are the prediction and centering curve search trajectories from \cref{line:curve:traj} of the \emph{curve} procedure, then we define the combined trajectory as $\hat{\omega} (\alpha) = \hat{\omega}^p (\alpha) + \hat{\omega}^c (1 - \alpha)$.
Note that $\alpha = 1$ corresponds to a full step in the adjusted prediction direction $\delta^p + \delta^{pt}$, and $\alpha = 0$ corresponds to a full step in the adjusted centering direction $\delta^c + \delta^{ct}$.

The \emph{comb} stepping procedure, which enhances the \emph{curve} stepping procedure by combining the prediction and centering phases, is summarized as follows.
Note that unlike the previous procedures, there is no parameter $j$ counting consecutive centering steps.
Occasionally in practice, the backtracking search on \cref{line:comb:alpha} below fails to find a positive step value, in which case we perform a centering step according to \crefrange{line:curve:traj}{line:curve:omega} of the \emph{curve} procedure.
\begin{algorithmic}[1]
\Procedure {CombStep}{$\omega^{i-1}$}
\State compute $\delta^c, \delta^{ct}, \delta^p, \delta^{pt}$ from \cref{eq:centsys,eq:centtoasys,eq:predsys,eq:predtoasys}
\Comment{use four directions instead of two}
\State let $\hat{\omega} (\alpha) \coloneqq \omega^{i-1} + \alpha ( \delta^p + \alpha \delta^{pt} ) + (1 - \alpha) ( \delta^c + (1 - \alpha) \delta^{ct} )$
\Comment{use combined trajectory}
\State $\hat{\alpha} \gets \max \{ \alpha \in \mathcal{A} : \piinf ( \hat{\omega} (\alpha) ) \leq \beta_2 \}$
\label{line:comb:alpha}
\State $\omega^i \gets \hat{\omega} (\hat{\alpha})$
\label{line:comb:omega}
\EndProcedure
\end{algorithmic}

%% file: cones.tex
\section{Oracles for predefined exotic cones}
\label{sec:cones}

Below we list 23 exotic cone types that we have predefined through Hypatia's generic cone interface (see \cref{sec:exotic}).
Each of these cones is represented in the benchmark set of conic instances that we introduce in \cref{sec:testing:bench}.
Recall that we write any exotic cone $\K$ in vectorized form, i.e. as a subset of $\bbR^q$, where $q = \dim(\K) \geq 1$ is the cone dimension.
For cones typically defined using symmetric matrices, we use the standard svec vectorization (see \cref{sec:prelim}) to ensure the vectorized cone is proper, to preserve inner products, and to simplify the dual cone definition.
Each cone is parametrized by at least one dimension and several cones have additional parameters such as numerical data.
For convenience, we drop these parameters from the symbols we use to represent cone types.
For several cones, we have implemented additional variants over complex numbers (for example, a Hermitian PSD cone), but we omit these definitions here for simplicity.
We defer a more complete description of Hypatia's exotic cones and LHSCBs to \citet{coey2021hypatia,coey2021self,kapelevich2021sum}.

\begin{description}
\item[Nonnegative cone.]
$\knn \coloneqq \bbR_{\geq}^d$ is the (self-dual) nonnegative real vectors (note for $d > 1$, $\knn$ is not a primitive cone).
\item[PSD cone.]
$\kpsd \coloneqq \bigl\{ w \in \bbR^{\sdim(d)} : \mat(w) \in \bbS^d_{\succeq} \bigr\}$ is the (self-dual) PSD matrices of side dimension $d$.
\item[Doubly nonnegative cone.]
$\kdnn \coloneqq \knn \cap \kpsd$ is the PSD matrices with all nonnegative entries of side dimension $d$.
\item[Sparse PSD cone.]
$\ksppsd$ is the PSD matrices of side dimension $s$ with a fixed sparsity pattern $\mathcal{S}$ containing $d \geq s$ nonzeros (including all diagonal elements); see \cref{sec:toos:sppsd}.
The dual cone $\ksppsd^\ast$ is the symmetric matrices with pattern $\mathcal{S}$ for which there exists a PSD completion, i.e. an assignment of the elements not in $\mathcal{S}$ such that the full matrix is PSD.
For simplicity, the complexity estimates in \cref{tab:oracles} assume the nonzeros are grouped under $J \geq 1$ supernodes, each containing at most $l$ nodes, and the monotone degree of each node is no greater than a constant $D$ \citep{andersen2013logarithmic}.
\item[Linear matrix inequality cone.]
$\klmi \coloneqq \bigl\{ w \in \bbR^d : \tsum{i \in \iin{d}} w_i P_i \in \bbS^s_{\succeq} \bigr\}$ are the vectors for which the matrix pencil of $d$ matrices $P_i \in \bbS^s, \forall i \in \iin{d}$ is PSD.
We assume $P_1 \succ 0$ so that we can use the initial interior point $e_1$.
\item[Infinity norm cone.]
$\klinf \coloneqq \{ (u, w) \in \bbR_{\geq} \times \bbR^d : u \geq \norm{w}_\infty \}$ is the epigraph of the $\ell_\infty$ norm on $\bbR^d$. 
Similarly, the dual cone $\klinf^\ast$ is the epigraph of the $\ell_1$ norm.
\item[Euclidean norm cone.]
$\kltwo \coloneqq \{ (u, w) \in \bbR_{\geq} \times \bbR^d : u \geq \lVert w \rVert \}$ is the (self-dual) epigraph of the $\ell_2$ norm on $\bbR^d$ (AKA second-order cone).
\item[Euclidean norm square cone.]
$\ksqr \coloneqq \{ (u, v, w) \in \bbR_{\geq} \times \bbR_{\geq} \times \bbR^d : 2 u v \geq \lVert w \rVert^2 \}$ is the (self-dual) epigraph of the perspective of the square of the $\ell_2$ norm on $\bbR^d$ (AKA rotated second-order cone).
\item[Spectral norm cone.]
$\klspec \coloneqq \{ (u, w) \in \bbR_{\geq} \times \bbR^{r s} : u \geq \sigma_1(\mat(w)) \}$, where $\sigma_1$ is the largest singular value function, is the epigraph of the spectral norm on $\bbR^{r \times s}$, assuming $r \leq s$ without loss of generality.
Similarly, $\klspec^\ast$ is the epigraph of the matrix nuclear norm (i.e.\ the sum of singular values).
\item[Matrix square cone.]
$\kmatsqr \coloneqq \bigl\{ (u, v, w) \in \bbR^{\sdim(r)} \times \bbR_{\geq} \times \bbR^{r s} : U \in \bbS^r_{\succeq}, 2 U v \succeq W W' \bigr\}$, where $U \coloneqq \mat(u)$ and $W \coloneqq \mat(w) \in \bbR^{r \times s}$, is the homogenized symmetric matrix epigraph of the symmetric outer product, assuming $r \leq s$ without loss of generality \citep{guler1998characterization}.
\item[Generalized power cone.]
$\kgpower \coloneqq \bigl\{ (u, w) \in \bbR^r_{\geq} \times \bbR^s : \tprod{i \in \iin{r}} u_i^{\alpha_i} \geq \lVert w \rVert \bigr\}$, parametrized by $\alpha \in \bbR^r_>$ with $e' \alpha = 1$, is the generalized power cone \citep[Section 3.1.2]{chares2009cones}.
\item[Power mean cone.]
$\kpower \coloneqq \bigl\{ (u, w) \in \bbR \times \bbR^d_{\geq} : u \leq \tprod{i \in \iin{d}} w_i^{\alpha_i} \bigr\}$, parametrized by exponents $\alpha \in \bbR^d_>$ with $e' \alpha = 1$, is the hypograph of the power mean on $\bbR_{\geq}^d$.
\item[Geometric mean cone.]
$\kgeom$ is the hypograph of the geometric mean on $\bbR_{\geq}^d$, a special case of $\kpower$ with equal exponents.
\item[Root-determinant cone.]
$\krtdet \coloneqq \bigl\{ (u, w) \in \bbR \times \bbR^{\sdim(d)} : W \in \bbS^d_{\succeq}, u \leq (\det(W))^{1/d} \bigr\}$, where $W \coloneqq \mat(w)$, is the hypograph of the $d$th-root-determinant on $\bbS_{\succeq}^d$.
\item[Logarithm cone.]
$\klog \coloneqq \cl \bigl\{ (u, v, w) \in \bbR \times \bbR_> \times \bbR^d_> : u \leq \tsum{i \in \iin{d}} v \log ( w_i / v ) \bigr\}$ is the hypograph of the perspective of the sum of logarithms on $\bbR_>^d$.
\item[Log-determinant cone.]
$\klogdet \coloneqq \cl \bigl\{ (u, v, w) \in \bbR \times \bbR_> \times \bbR^{\sdim(d)} : 
W \in \bbS^d_{\succ}, u \leq v \logdet( W / v ) \bigr\}$, where $W \coloneqq \mat(w)$, is the hypograph of the perspective of the log-determinant on $\bbS_{\succ}^d$.
\item[Separable spectral function cone.]
$\ksepspec \coloneqq \cl \{ (u, v, w) \in \bbR \times \bbR_> \times \intr(\cQ) : u \geq v \varphi (w / v) \}$, where $\cQ$ is $\knn$ or $\kpsd$ (a cone of squares of a Jordan algebra), is the epigraph of the perspective of a convex separable spectral function $\varphi : \intr(\cQ) \to \bbR$, such as the sum or trace of the negative logarithm, negative entropy, or power in $(1, 2]$ (see \citet{coey2021self} for more details).
The complexity estimates in \cref{tab:oracles} depend on whether $\cQ$ is $\knn$ or $\kpsd$.
\item[Relative entropy cone.]
$\krelentr \coloneqq \cl \big\{ (u, v, w) \in \bbR \times \bbR^d_> \times \bbR^d_> : \allowbreak u \geq \tsum{i \in \iin{d}} w_i \log(w_i / v_i) \big\}$ is the epigraph of vector relative entropy.
\item[Matrix relative entropy cone.]
$\K_{\matrelentr} \coloneqq \cl \bigl\{ (u, v, w) \in \bbR \times \bbR^{\sdim(d)} \times \bbR^{\sdim(d)} : V \in \bbS^d_{\succ}, W \in \bbS^d_{\succ}, u \geq \tr(W (\log(W) - \log(V))) \bigr\}$, where $V \coloneqq \mat(v)$ and $W \coloneqq \mat(w)$, is the epigraph of matrix relative entropy.\footnote{
The logarithmically homogeneous barrier for $\K_{\matrelentr}$ that Hypatia uses is conjectured by \citet{karimi2020domaindriven} to be self-concordant.
}
\item[Weighted sum-of-squares (WSOS) cones.]
An interpolant basis represents a polynomial implicitly by its evaluations at a fixed set of $d$ points.
Given a basic semialgebraic domain defined by $r$ polynomial inequalities, the four WSOS cones below are parameterized by matrices $P_l \in \bbR^{d \times s_l}$ for $l \in \iin{r}$.
Each $P_l$ is constructed by evaluating $s_l$ independent polynomials (columns) at the $d$ points (rows), following \citet{papp2019sum}.
For simplicity, the complexity estimates in \cref{tab:oracles} assume $s_l = s, \forall l \in \iin{r}$.
Note that $s < d \leq s^2$.
We define $\kwsos$ and $\kmatwsos$ in \citet{coey2020solving}, and $\klonewsos$ and $\kltwowsos$ in \citet[Equations 2.7 and 4.10]{kapelevich2021sum}. 
\begin{description}
\item[Scalar WSOS cone.]
$\kwsos$ is a cone of polynomials that are guaranteed to be nonnegative pointwise on the domain.
\item[Symmetric matrix WSOS cone.]
$\kmatwsos$ is a cone of polynomial symmetric matrices (in an svec-like format) of side dimension $t$ that are guaranteed to belong to $\bbS_{\succeq}$ pointwise on the domain. 
We let $m \coloneqq s t + d$ in \cref{tab:oracles} for succinctness.
\item[$\ell_1$ epigraph WSOS cone.]
$\klonewsos$ is a cone of polynomial vectors of length $1 + t$ that are guaranteed to belong to $\klinf^\ast$ pointwise on the domain.
\item[$\ell_2$ epigraph WSOS cone.]
$\kltwowsos$ is a cone of polynomial vectors of length $1 + t$ that are guaranteed to belong to $\kltwo$ pointwise on the domain.
\end{description}
\end{description}

For each cone, we have an analytic form for the feasibility check, gradient, Hessian, and TOO oracles defined in \cref{sec:exotic}.
That is, we always avoid iterative numerical procedures such as optimization, which are typically slow, numerically unstable, and require tuning.
Hypatia's algorithm always evaluates the feasibility check before the gradient, Hessian, and TOO (which are only defined at strictly feasible points), and the gradient is evaluated before the Hessian and TOO.
For most of these cones, the feasibility check and gradient oracles compute values and factorizations that are also useful for computing the Hessian and TOO, so this data is cached in the cone data structures and re-used where possible.
In \cref{tab:oracles}, we estimate the time complexities (ignoring constants) of these four oracles for each cone, counting the cost of cached values and factorizations only once (for the oracle that actually computes them).
\Cref{tab:oracles} shows that the TOO is never more expensive than the feasibility check, gradient, and Hessian oracles (i.e.\ the oracles needed by SY).
Indeed, our computational results in \cref{sec:testing:results} demonstrate that the TOO is very rarely an algorithmic bottleneck in practice.

\begin{table}[!htb]
\input{tables/oracles} 
\caption{
Cone dimension $\dim(\K)$, LHSCB parameter $\nu$, and time complexity estimates (ignoring constants) for our feasibility check, gradient, Hessian, and TOO implementations, for the exotic cones defined in \cref{sec:cones}.
}
\label{tab:oracles}
\end{table}

Our TOO in \cref{eq:too} is distinct from the `higher order corrector' terms proposed by \citet{mehrotra1992implementation,dahl2021primal}.
The method by \citet{mehrotra1992implementation} only applies to symmetric cones, and \citet{dahl2021primal} test their technique only for the standard exponential cone.
Compared to the third order term proposed by \citet{dahl2021primal}, our TOO has a simpler and more symmetric structure, as it relies on only one direction $\delta_{\bs}$ rather than two.
Like the gradient and Hessian oracles, our TOO is additive for sums of LHSCBs, which can be useful for cones (such as $\kdnn$ and $\kwsos$) that are defined as intersections of other cones.
We leverage these properties to obtain fast and numerically stable TOO implementations.

To illustrate, in \cref{sec:toos:psdslice} we define LHSCBs and derive efficient TOO procedures for a class of cones that can be characterized as intersections of slices of the PSD cone $\kpsd$.
We consider $\klmi$ in \cref{sec:toos:lmi} and $\kwsos^\ast$ and $\kmatwsos^\ast$ in \cref{sec:toos:wsos}. 
In \cref{sec:toos:sppsd}, we handle $\ksppsd$ by differentiating a procedure by \citet{andersen2013logarithmic} for computing Hessian products.
In \cref{sec:toos:eucl} we also show how to compute the TOO for $\kltwo$ and $\ksqr$.
In \citet{coey2021self}, we derive efficient TOO procedures for a class of spectral function cones on positive domains ($\ksepspec$, $\klog$, $\klogdet$, $\kgeom$, $\krtdet$).

%% file: tables/oracles.tex
\centering
\begin{tabular}{ccccccc}
\toprule
cone & $\dim(\K)$ & $\nu$ & feasibility & gradient & Hessian & TOO \\ 
\midrule
$\knn$ & $d$ & $d$ &
$d$ & $d$ &
$d$ & $d$
\\
$\kpsd$ & $\sdim(d)$ & $d$ &
$d^3$ & $d^3$ &
$d^4$ & $d^3$
\\
$\kdnn$ & $\sdim(d)$ & $\sdim(d)$ &
$d^3$ & $d^3$ &
$d^4$ & $d^3$ 
\\
$\ksppsd$ & $d$ & $s$ &
$J D^2 l$ & $J D^2 l$ &
$d J D^2 l$ & $J D^2 l$
\\
$\klmi$ & $d$ & $s$ &
$d s^2 + s^3$ & $d s^3$ &
$d^2 s^2$ & $d s^2 + s^3$ 
\\
$\klinf$ & $1 + d$ & $1 + d$ &
$d$ & $d$ &
$d$ & $d$ 
\\
$\kltwo$, $\ksqr$ & $1 + d$ & $2$ &
$d$ & $d$ &
$d^2$ & $d$
\\
$\klspec$ & $1 + r s$ & $1 + r$ &
$r^2s + r^3$ & $r^2s + r^3$ &
$r^2 s^2$ & $r s^2$
\\
$\kmatsqr$ & $\sdim(r) + 1 + r s$ & $1 + r$ &
$r^2 s + r^3$ & $r^2 s + r^3$ &
$r^2 s^2$ & $r s^2$ 
\\
$\kgpower$ & $r + s$ & $1 + r$ &
$r + s$ & $r + s$ &
$r^2 + s^2$ & $r + s$ 
\\
$\kpower$, $\kgeom$ & $1 + d$ & $1 + d$ &
$d$ & $d$ &
$d^2$ & $d$ 
\\
$\krtdet$ & $1 + \sdim(d)$ & $1 + d$ &
$d^3$ & $d^3$ &
$d^4$ & $d^3$ 
\\
$\klog$ & $2 + d$ & $2 + d$ &
$d$ & $d$ &
$d^2$ & $d$ 
\\
$\klogdet$ & $2 + \sdim(d)$ & $2 + d$ &
$d^3$ & $d^3$ &
$d^4$ & $d^3$ 
\\
$\ksepspec \mhyphen \knn$ & $2 + d$ & $2 + d$ &
$d$ & $d$ &
$d^2$ & $d$
\\
$\ksepspec \mhyphen \kpsd$ & $2 + \sdim(d)$ & $2 + d$ &
$d^3$ & $d^3$ &
$d^5$ & $d^3$
\\
$\krelentr$ & $1 + 2d$ & $1 + 2 d$ &
$d$ & $d$ &
$d^2$ & $d$ 
\\
$\kmatrelentr$ & $1 + 2 \sdim(d)$ & $1 + 2 d$ &
$d^3$ & $d^3$ &
$d^5$ & $d^4$
\\
$\kwsos$ & $d$ & $s r$ &
$d s^2 r $ & $d s^2 r$ &
$d^2 s r$ & $d s^2 r$ 
\\
$\kmatwsos$ & $d \sdim(t)$ & $s t r$ &
$m s^2 t^2 r$ & $d s^2 t^2 r$ &
$d^2 s t^3 r$ & $m s^2 t^2 r$ 
\\
$\klonewsos$ & $d (1 + t)$ & $s t r$ &
$d s^2 t r$ & $d s^2 t r$ &
$d^2 s t r$ & $d s^2 t r$ 
\\
$\kltwowsos$ & $d (1 + t)$ & $2 s r$ &
$d s^2 t r$ & $d s^2 t r$ &
$d^2 s t^2 r$ & $d s^2 t^2 r$
\\
\bottomrule
\end{tabular}

%% file: testing.tex
\section{Computational testing}
\label{sec:testing}

In \cref{sec:testing:bench}, we introduce a diverse set of exotic conic benchmark instances generated from a variety of applied examples.
In \cref{sec:testing:method}, we describe our methodology for comparing the stepping procedures from \cref{sec:alg:step}, and in \cref{sec:testing:results} we examine our computational results.

\subsection{Exotic conic benchmark set}
\label{sec:testing:bench}

We generate 379 instances (in our primal general form \cref{eq:prim}) from 37 applied examples in Hypatia's examples folder.
All instances are primal-dual feasible except for 12 that are primal infeasible and one that is dual infeasible.
For most examples, we construct multiple formulations using different predefined exotic cones from the list in \cref{sec:cones}.
Each cone from this list appears in at least one instance, so we consider our benchmark set to be the most diverse collection of conic instances available.

We generate most instances using JuMP, but for some we use Hypatia's native model interface.
Due to the size of some instances and the lack of a standard instance storage format recognizing our cone types, we generate all instances on the fly in Julia.
For instances that use random data, we set random seeds to ensure reproducibility.
\Cref{fig:npqKhist} shows the distributions of instance dimensions and exotic cone counts.
All instances have at least one cone (note any $\knn$ cones are concatenated together, so $\knn$ is counted at most once) and take at least one iteration to solve with Hypatia.

Below we briefly introduce each example. 
In \cref{tab:examples}, we summarize for each example the number of corresponding instances and the cone types represented in at least one of the instances.
We do not distinguish dual cones and primal cones in this summary (for example, instances that use $\klinf^\ast$ are only listed as using $\klinf$).
For some examples, we describe a subset of formulations in \citet{coey2020solving,kapelevich2021sum}.
Our benchmark set includes ten instances from CBLIB (a conic benchmark instance library, see \citet{friberg2016cblib}).
We chose to avoid running a larger sample of instances from CBLIB so that the relatively few cone types supported by CBLIB version 3 are not over-represented in our benchmark set.
\begin{description}
\item[Central polynomial matrix.]
Minimize a spectral function of a gram matrix of a polynomial. 
\item[Classical-quantum capacity.]
Compute the capacity of a classical-to-quantum channel (adapted from \citet[Section 3.1]{fawzi2018efficient}).
\item[Condition number.]
Minimize the condition number of a matrix pencil subject to a linear matrix inequality (adapted from \citet[Section 3.2]{boyd1994linear}).
\item[Contraction analysis.]
Find a contraction metric that guarantees global stability of a dynamical system (adapted from \citet[Section 5.3]{aylward2008stability}).
Six instances are primal infeasible.
\item[Convexity parameter.]
Find the strong convexity parameter of a polynomial function over a domain.
\item[Covariance estimation.]
Estimate a covariance matrix that satisfies some given prior information and minimizes a given convex spectral function.
\item[Density estimation.]
Find a valid polynomial density function maximizing the likelihood of a set of observations (compare to \cite[Section 4.3]{papp2014shape}, see \citet[Section 5.6]{coey2020solving}).
\item[Discrete maximum likelihood.] 
Maximize the likelihood of some observations at discrete points, subject to the probability vector being close to a uniform prior.
\item[D-optimal design.]
Solve a D-optimal experiment design problem, i.e.\ maximize the determinant of the information matrix subject to side constraints (adapted from \citet[Section 7.5]{boyd2004convex}; see \citet[Section 5.4]{coey2020solving}).
\item[Entanglement-assisted capacity.]
Compute the entanglement-assisted classical capacity of a quantum channel (adapted from \citet[Section 3.2]{fawzi2018efficient}). 
\item[Experiment design.]
Solve a general experiment design problem that minimizes a given convex spectral function of the information matrix subject to side constraints (adapted from \citet[Section 7.5]{boyd2004convex}).
\item[Linear program.]
Solve a simple linear program.
\item[Lotka-Volterra.]
Find an optimal controller for a Lotka-Volterra model of population dynamics (adapted from \citet[Section 7.2]{korda2016controller}).
\item[Lyapunov stability.]
Minimize an upper bound on the root mean square gain of a dynamical system (adapted from \citet[Section 6.3.2]{boyd1994linear} and \citet[Page 6]{boyd2009lecture}).
\item[Matrix completion.]
Complete a rectangular matrix by minimizing the nuclear norm and constraining the missing entries (compare to \citet[Equation 8]{agrawal2019dgp}; see \citet[Section 5.2]{coey2020solving}).
\item[Matrix quadratic.]
Find a rectangular matrix that minimizes a linear function and satisfies a constraint on the outer product of the matrix.
\item[Matrix regression.]
Solve a multiple-output (or matrix) regression problem with regularization terms, such as $\ell_1$, $\ell_2$, or nuclear norm (see \citet[Section 5.3]{coey2020solving}). 
\item[Maximum volume hypercube.]
Find a maximum volume hypercube (with edges parallel to the axes) inside a given polyhedron or ellipsoid (adapted from \citet[Section 4.3.2]{mosek2020modeling}).
\item[Nearest correlation matrix.]
Compute the nearest correlation matrix in the quantum relative entropy sense (adapted from \citet{fawzi2019semidefinite}).
\item[Nearest polynomial matrix.]
Given a symmetric matrix of polynomials $H$, find a polynomial matrix $Q$ that minimizes the sum of the integrals of its elements over the unit box and guarantees $Q - H$ is pointwise PSD on the unit box.
\item[Nearest PSD matrix.]
Find a sparse PSD matrix or a PSD-completable matrix (with a given sparsity pattern) with constant trace that maximizes a linear function (adapted from \citet{sun2015decomposition}).
\item[Nonparametric distribution.]
Given a random variable taking values in a finite set, compute the distribution minimizing a given convex spectral function over all distributions satisfying some prior information.
\item[Norm cone polynomial.]
Given a vector of polynomials, check a sufficient condition for pointwise membership in $\kltwo$ or $\klinf^\ast$.
Four instances are primal infeasible.
\item[Polynomial envelope.]
Find a polynomial that closely approximates, over the unit box, the lower envelope of a given list of polynomials (see \citet[Section 7.2.1]{papp2019sum}).
\item[Polynomial minimization.]
Compute a lower bound for a given polynomial over a given semialgebraic set (see \citet[Section 7.3.1]{papp2019sum} and \citet[Section 5.5]{coey2020solving}).
Some instances use polynomials with known optimal values from \citet{burkardt2021polynomials}. 
\item[Polynomial norm.]
Find a polynomial that, over the unit box, has minimal integral and belongs pointwise to the epigraph of the $\ell_1$ or $\ell_2$ norm of other given polynomials (see \citet[Section 6]{kapelevich2021sum}).
\item[Portfolio.]
Maximize the expected returns of a stock portfolio and satisfy various risk constraints (see \citet[Section 5.1]{coey2020solving}).
\item[Region of attraction.]
Find the region of attraction of a polynomial control system (see \citet[Section 9.1]{henrion2013convex}).
\item[Relative entropy of entanglement.]
Compute a lower bound on relative entropy of entanglement with a positive partial transpose relaxation (adapted from \citet[Section 4]{fawzi2018efficient}).
\item[Robust geometric programming.]
Bound the worst-case optimal value of an uncertain signomial function with a given coefficient uncertainty set (adapted from \citet[Equation 39]{chandrasekaran2017relative}).
\item[Semidefinite polynomial matrix.]
Check a sufficient condition for global convexity of a given polynomial.
Two instances are primal infeasible and one is dual infeasible.
\item[Shape constrained regression.]
Given a dataset, fit a polynomial function that satisfies shape constraints such as monotonicity or convexity over a domain (see \citet[Section 5.7]{coey2020solving}).
Several instances use real datasets from \citet{mazumder2019computational}.
\item[Signomial minimization.]
Compute a global lower bound for a given signomial function (see \citet{murray2020signomial}).
Several instances use signomials with known optimal values from \citet{murray2020signomial,chandrasekaran2016relative}.
\item[Sparse LMI.]
Optimize over a simple linear matrix inequality with sparse data.
\item[Sparse principal components.]
Solve a convex relaxation of the problem of approximating a symmetric matrix by a rank-one matrix with a cardinality-constrained eigenvector (adapted from \citet[Section 2]{d2007direct}).
\item[Stability number.]
Given a graph, solve for a particular strengthening of the theta function towards the stability number (adapted from \citet[Equation 2.4]{laurent2015conic}).
\end{description}

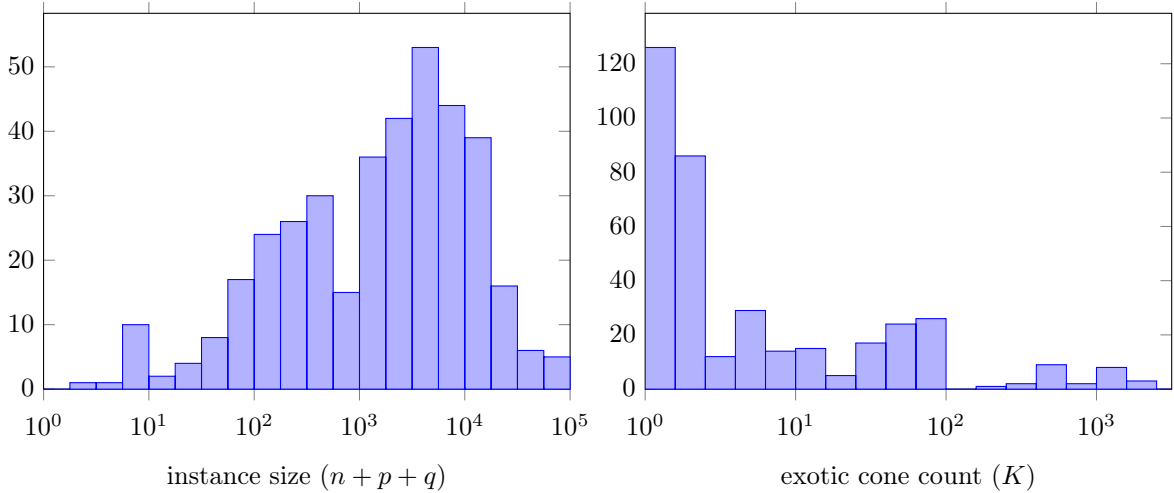
\begin{figure}[!htb]
\input{tikz/npqKhist}
\caption{
Histograms summarizing the benchmark instances in the primal conic form \cref{eq:prim}.
Instance size (log scale) is the sum of the primal variable, equality, and conic constraint dimensions.
Exotic cone count (log scale) is the number of exotic cones comprising the Cartesian product cone.
}
\label{fig:npqKhist}
\end{figure}

\begin{table}[!htbp]
\input{tables/examples}
\caption{
For each example, the count of instances and list of exotic cones (defined in \cref{sec:cones}) used in at least one instance.
}
\label{tab:examples}
\end{table}

\subsection{Methodology}
\label{sec:testing:method}

We can assess the practical performance of a stepping procedure on a given benchmark instance according to several metrics: whether the correct conic certificate (satisfying our numerical tolerances, discussed below) is found, and if so, the PDIPM iteration count and solve time.
Across the benchmark set, we compare performance between consecutive pairs of the five stepping procedures outlined in \cref{sec:alg:step}.
\begin{description}
\item[basic.]
The basic prediction or centering stepping procedure without any enhancements; described in \cref{sec:alg:step:basic}, this is similar to the method in Alfonso solver \citep{papp2021alfonso}, which is a practical implementation of the algorithm by \citet{skajaa2015homogeneous,papp2017homogeneous}.
\item[prox.]
The \emph{basic} procedure modified to use a less restrictive central path proximity condition; described in \cref{sec:alg:step:prox}.
\item[TOA.]
The \emph{prox} procedure with the TOA enhancement to incorporate third order LHSCB information; described in \cref{sec:alg:step:toa}.
\item[curve.]
The \emph{TOA} procedure adapted for a single backtracking search on a curve instead of two backtracking line searches; described in \cref{sec:alg:step:curve}.
\item[comb.]
The \emph{curve} procedure modified to search along a curve of combinations of both the prediction and centering directions and their corresponding adjustment directions; described in \cref{sec:alg:step:comb}.
\end{description}

We perform all instance generation, computational experiments, and results analysis using double precision floating point format, with Ubuntu 21.04, Julia 1.7, and Hypatia 0.5.1 (with default options), on dedicated hardware with an AMD Ryzen 9 3950X 16-core processor (32 threads) and 128GB of RAM.
In \cref{sec:linalg}, we outline the default procedures Hypatia uses for preprocessing, initial point finding, and linear system solving for search directions.
Simple scripts and instructions for reproducing all results are available in Hypatia's benchmarks/stepper folder.
The benchmark script runs all solves twice and uses results from the second run, to exclude Julia compilation overhead.
A CSV file containing raw results is available at the Hypatia wiki page.

When Hypatia converges for an instance, i.e.\ claims it has found a certificate of optimality, primal infeasibility, or dual infeasibility, our scripts verify that this is the correct type of certificate for that instance.
For some instances, our scripts also check additional conditions, for example that the objective value of an optimality certificate approximately equals the known true optimal value.
We do not set restrictive time or iteration limits.
All failures to converge are caused by Hypatia `stalling' during the stepping iterations: either the backtracking search cannot step a distance of at least the minimal value in the $\alpha$ schedule, or across several prediction steps or combined directions steps, Hypatia fails to make sufficient progress towards meeting the convergence conditions in \cref{sec:alg:alg}.

Since some instances are more numerically challenging than others, we set the termination tolerances (described in \cref{sec:alg:alg}) separately for each instance.
Let $\epsilon \approx 2.22 \times 10^{-16}$ be the machine epsilon.
For most instances, we use $\varepsilon_f = \varepsilon_r = 10 \epsilon^{1/2} \approx 1.49 \times 10^{-7}$ for the feasibility and relative gap tolerances, $\varepsilon_i = \varepsilon_a = 10 \epsilon^{3/4} \approx 1.82 \times 10^{-11}$ for the infeasibility and absolute gap tolerances, and $\varepsilon_p = 0.1 \epsilon^{3/4} \approx 1.82 \times 10^{-13}$ for the ill-posedness tolerance.
For 50 instances that are particularly numerically challenging, we loosen all of these tolerances by a factor of either 10 or 100, and for two challenging primal infeasible instances of the \emph{contraction analysis} example, we set $\varepsilon_i = 10^{-9}$.
This ensures that for every benchmark instance, at least one of the five stepping procedures converges.

Following \citet{fleming1986not}, we define the \emph{shifted geometric mean} with shift $s \geq 0$, for $d$ values $v \in \bbR_>^d$, as:
\begin{equation}
M(v,s) \coloneqq 
\textstyle\prod_{i \in \iin{d}} (v_i + s)^{1/d} - s.
\label{eq:shiftgeomean}
\end{equation}
We always apply a shift of one for iteration counts. 
Since different stepping procedures converge on different subsets of instances, in tables we show three types of shifted geometric means, each computed from a vector of values ($v$ in \cref{eq:shiftgeomean}) obtained using one of the following approaches.
\begin{description}
\item[every.]
Values for the 353 instances on which every stepping procedure converged. 
\item[this.]
Values for instances on which this stepping procedure (corresponding to the row of the table) converged.
\item[all.]
Values for all instances, but for any instances for which this stepping procedure (corresponding to the row of the table) failed to converge, the value is replaced with two times the maximum value for that instance across the stepping procedures that converged.
\end{description}
The shifted geometric means for the \emph{every} approach are the most directly comparable because they are computed on a fixed subset of instances, so we usually quote the \emph{every} results in our discussion in \cref{sec:testing:results}.

\Cref{tab:agg} shows counts of converged instances and shifted geometric means of iteration count and total solve time (in milliseconds), for the five stepping procedures.
We use a shift of one millisecond for the solve times in \cref{tab:agg}, as some instances solve very quickly (see \cref{fig:solvehist}).

\Cref{tab:subtime} shows shifted geometric means of the time (in milliseconds) Hypatia spends performing each of the following key algorithmic components, for the five stepping procedures.
\begin{description}
\item[init.]
Performed once during an entire solve run, independently of the stepping iterations.
Includes rescaling and preprocessing of model data, initial interior point finding, and linear system solver setup (see \cref{sec:linalg}).
\item[LHS.]
Performed at the start of each iteration.
Includes updating data that the linear system solver (which has a fixed LHS in each iteration) uses to efficiently compute at least one direction (such as updating and factorizing the positive definite matrix in \cref{sec:linalg}).
\item[RHS.]
Performed between one and four times per iteration, depending on the stepping procedure.
Includes updating an RHS vector (see \cref{eq:dirs}) for the linear system for search directions.
Note that the TOO is only evaluated while computing the centering TOA RHS \cref{eq:centtoasys:K} and the prediction TOA RHS \cref{eq:predtoasys:K}.
\item[direc.]
Performed for each RHS vector.
Includes solving the linear system for a search direction (see \cref{eq:dirs}) using the data computed during \emph{LHS} and a single RHS vector computed during \emph{RHS}, and performing iterative refinement on the direction (see \cref{sec:linalg}).
\item[search.]
Performed once or twice per iteration (occasionally more if the step length is near zero), depending on the stepping procedure.
Includes searching using backtracking along a line or curve to find an interior point satisfying the proximity conditions (see \cref{sec:search}).
\end{description}
For some instances that solve extremely quickly, these subtimings sum to only around half of the total solve time due to extraneous overhead.
However for slower instances, these components account for almost the entire solve time.
In \cref{tab:subtime}, \emph{total} is the time over all iterations, and \emph{per iteration} is the average time per iteration (the arithmetic means are computed before the shifted geometric mean).
We use a shift of 0.1 milliseconds for the \emph{init} and \emph{total} subtimings (left columns) and a shift of 0.01 milliseconds for the \emph{per iteration} subtimings (right columns).

Finally, in \cref{fig:perftotal,fig:perf} we use \emph{performance profiles} \citep{dolan2002benchmarking,gould2016note} to compare iteration counts and solve times between pairs of stepping procedures.
These should be interpreted as follows.
The \emph{performance ratio} for procedure $i$ and instance $j$ is the value (iterations or solve time) attained by procedure $i$ on instance $j$ divided by the better/smaller value attained by the two procedures on instance $j$.
Hence a performance ratio is at least one, and smaller values indicate better relative performance.
For a point $(x,y)$ on a performance profile curve for a particular procedure, $x$ is the logarithm (base $2$) of performance ratio and $y$ is the proportion of instances for which the procedure attains that performance ratio or better/smaller.
For example, a curve crosses the vertical axis at the proportion of instances on which the corresponding procedure performed at least as well as the alternative procedure.
We use the Julia package BenchmarkProfiles.jl \citep{juliasmoothoptimizers2021benchmark} to compute coordinates for the performance profile curves.

\subsection{Results}
\label{sec:testing:results}

\Cref{tab:agg,fig:perf} demonstrate that each of the four cumulative stepping enhancements tends to improve Hypatia's iteration count and solve time.
The enhancements do not have a significant impact on the number of instances Hypatia converges on. 
However, if we had enforced time or iteration limits, the enhancements would have also improved the number of instances solved.
This is clear from \cref{fig:solvehist}, which shows the distributions of iteration counts and solve times for the \emph{basic} and \emph{comb} stepping procedures.
We note that \cref{fig:comb} (left) supports the intuition that formulation size is strongly positively correlated with solve time for \emph{comb}.

Overall, \cref{tab:agg} shows that on the subset of instances solved by every stepping procedure (\emph{every}), the enhancements together reduce the shifted geometric means of iterations and solve time by more than 80\% and 70\% respectively (i.e.\ comparing \emph{comb} to \emph{basic}). 
\Cref{fig:perftotal} shows that the iteration count and solve time improve on nearly every instance solved by both \emph{basic} and \emph{comb}, and the horizontal axis scale shows that the magnitude of these improvements is large on most instances. 
\Cref{fig:relimpr} shows that for instances that take more iterations or solve time, the enhancements tend to yield a greater improvement in these measures.
On every instance, the enhancements improve the iteration count by at least 33\%.
The few instances for which solve time regressed with the enhancements all solve relatively quickly.

\begin{figure}[!htb]
\input{tikz/solvehist}
\caption{
Overlayed histograms of iteration count (left, log scale) and solve time (right, log scale, in seconds) for the \emph{basic} and \emph{comb} stepping procedures, excluding instances that fail to converge.
}
\label{fig:solvehist}
\end{figure}
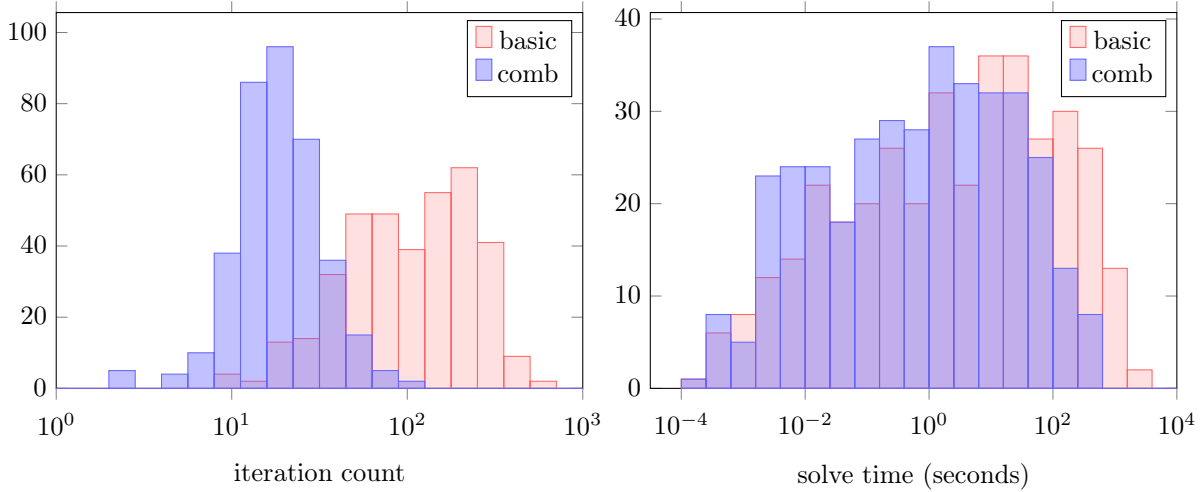

\begin{table}[!htb]
\input{tables/agg}
\caption{
For each stepping procedure, the number of converged instances and shifted geometric means of iterations and solve times (in milliseconds).
}
\label{tab:agg}
\end{table}

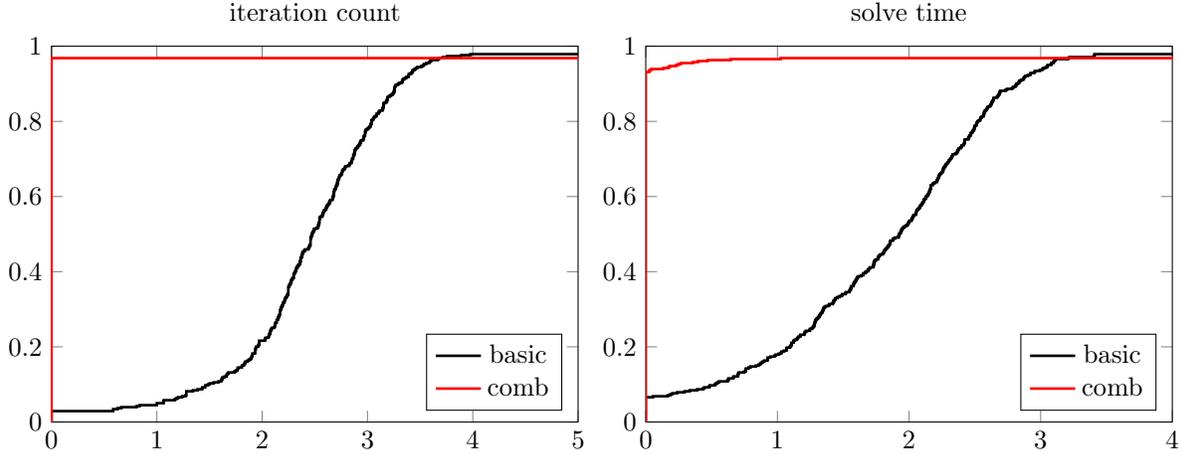
\begin{figure}[!htb]
\input{tikz/perftotal}
\caption{
Performance profiles (see \cref{sec:testing:method}) of iteration count (left) and solve time (right) for the four stepping enhancements overall.
}
\label{fig:perftotal}
\end{figure}

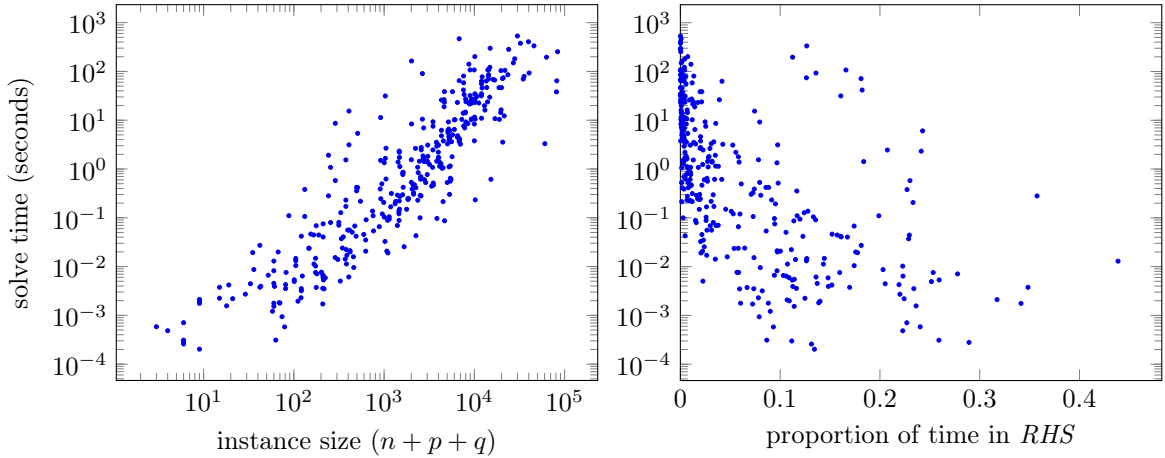
\begin{figure}[!htb]
\input{tikz/comb}
\caption{
Solve time (log scale, in seconds) for the \emph{comb} stepping procedure against (left) instance size (log scale) and (right) the proportion of solve time spent in \emph{RHS}, excluding instances that fail to converge.
}
\label{fig:comb}
\end{figure}

\begin{table}[!htb]
\input{tables/subtime}
\caption{
For each stepping procedure, the shifted geometric means of subtimings (in milliseconds) for the key algorithmic components.
}
\label{tab:subtime}
\end{table}

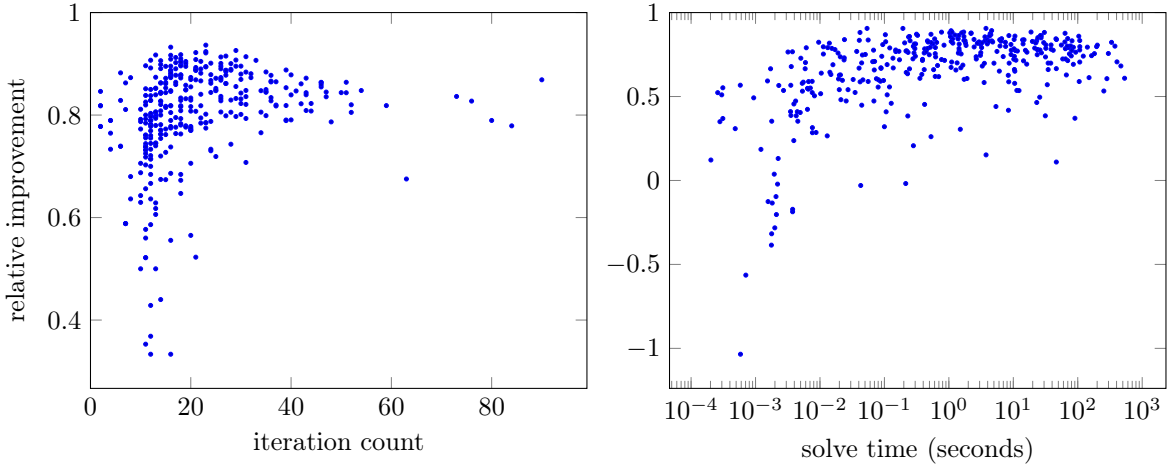
\begin{figure}[!htb]
\input{tikz/relimpr}
\caption{
Relative improvement, from \emph{basic} to \emph{comb}, in iteration count (left) or solve time (right) against iteration count or solve time (in seconds) respectively for \emph{comb}, over the 356 instances on which both \emph{basic} and \emph{comb} converge.
}
\label{fig:relimpr}
\end{figure}

\begin{figure}[!p]
\input{tikz/perf}
\caption{
Performance profiles (see \cref{sec:testing:method}) of iteration count (left column) and solve time (right column) for the four stepping enhancements (rows).
}
\label{fig:perf}
\end{figure}
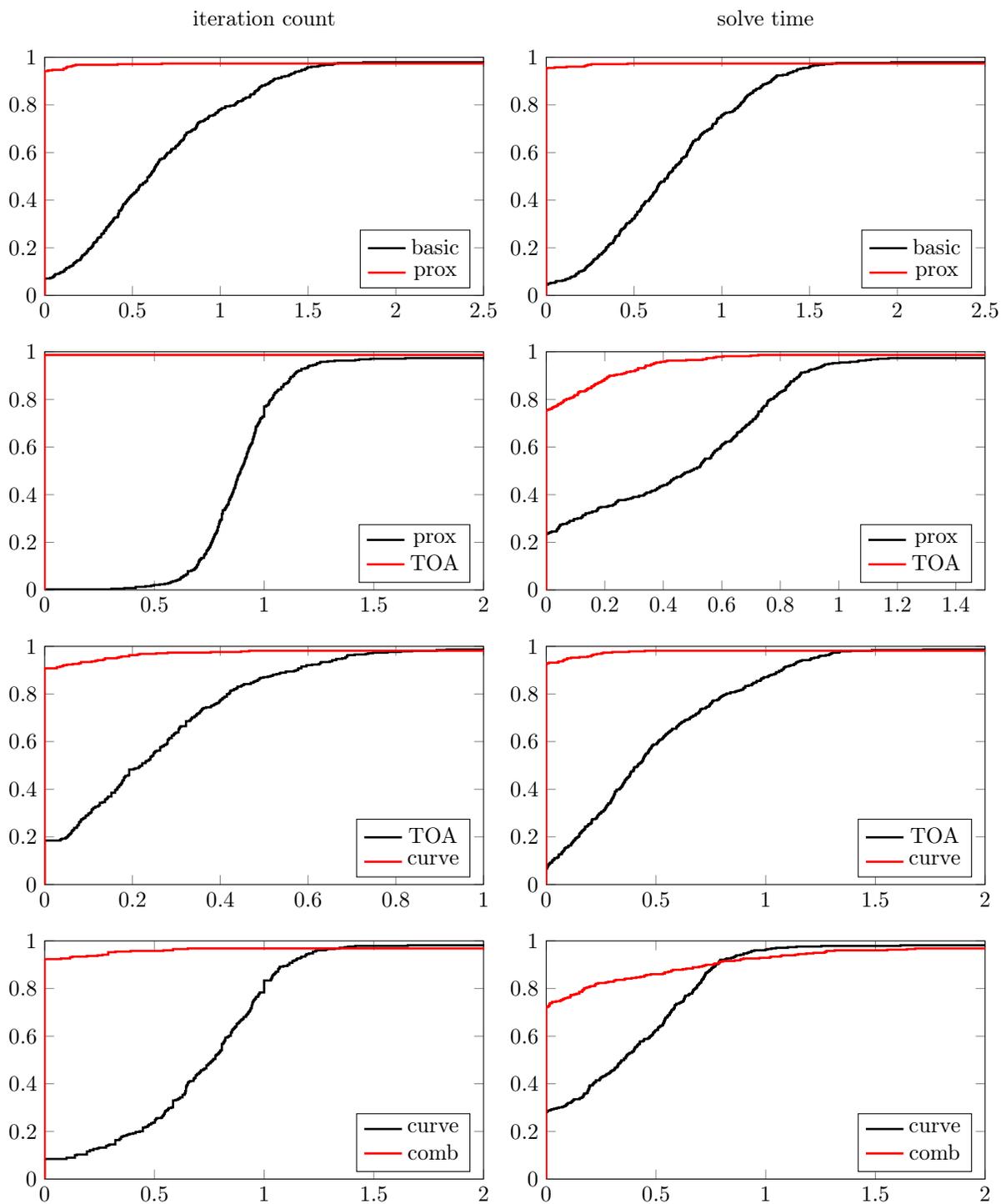

Each enhancement, by design, changes one modular component or aspect of the stepping procedure.
Below, we examine the impact of our algorithmic choices by discussing pairwise comparisons of consecutive stepping procedures.


\subsubsection{Less restrictive proximity}
\label{sec:testing:prox}

We compare \emph{basic} and \emph{prox} to evaluate the central path proximity enhancement introduced in \cref{sec:alg:step:prox}.
\Cref{fig:perf} (first row) shows that the iteration count and solve time improve for nearly all instances.
From \cref{tab:agg}, the shifted geometric means of iteration count and solve time improve by over 35\%.

The similarity between the iteration count and solve time performance profiles in \cref{fig:perf} and also between the per iteration subtimings in \cref{tab:subtime} suggests that the solve time improvement is driven mainly by the reduction in iteration count.
The per iteration \emph{search} time decreases slightly, since on average fewer backtracking search steps are needed per iteration for \emph{prox} (because it tends to step further in the prediction directions, as evidenced by the smaller iteration counts).
These results suggest that the central path proximity restrictions in the algorithms by \citet{skajaa2015homogeneous,papp2021alfonso} are too conservative from the perspective of practical performance, and that we need not restrict iterates to a very small neighborhood of the central path in order to obtain high quality prediction directions in practice.

\subsubsection{Third order adjustments}
\label{sec:testing:toa}

We compare \emph{prox} and \emph{TOA} to evaluate the TOA enhancement introduced in \cref{sec:alg:step:toa}.
\Cref{fig:perf} (second row) shows that the iteration count improves for all instances and by a fairly consistent magnitude, and the solve time improves for nearly 80\% of instances.
From \cref{tab:agg}, the shifted geometric means of iteration count and solve time improve by over 45\% and over 20\% respectively.

Since \emph{TOA} computes an additional direction and performs an additional backtracking search every iteration, the per iteration times for \emph{direc} and \emph{search} in \cref{tab:subtime} nearly double.
The \emph{RHS} time increases substantially, because the TOO is evaluated for the second RHS vector (used to compute the TOA direction), but \emph{RHS} is still much faster than the other components.
Per iteration, \emph{direc} and \emph{search} also remain fast compared to \emph{LHS}.
We see an overall solve time improvement because the reduction in iteration count usually outweighs the additional cost at each iteration.
This suggests that the TOO is generally relatively cheap to compute, and our TOA approach very reliably improves the quality of the search directions.

\subsubsection{Curve search}
\label{sec:testing:curve}

We compare \emph{TOA} and \emph{curve} to evaluate the curve search enhancement introduced in \cref{sec:alg:step:curve}.
\Cref{fig:perf} (third row) shows that the iteration count and solve time improve for most instances, with larger and more consistent improvements for the solve time.
From \cref{tab:agg}, the shifted geometric means of iteration count and solve time improve by over 15\% and over 25\% respectively.

Since \emph{curve} performs one backtracking search along a curve instead of the two backtracking line searches needed by \emph{TOA}, the per iteration \emph{search} time in \cref{tab:subtime} nearly halves.
The other subtimings are unaffected, so \emph{curve} improves the speed of each iteration.
The improvement in iteration count may stem from the more dynamic nature of the curve search compared to \emph{TOA}'s approach of computing a fixed combination of the unadjusted and TOA directions as a function of the step distance in the unadjusted direction.

\subsubsection{Combined directions}
\label{sec:testing:comb}

Finally, we compare \emph{curve} and \emph{comb} to evaluate the combined directions enhancement introduced in \cref{sec:alg:step:comb}.
\Cref{fig:perf} (fourth row) shows that the iteration count and solve time improve on around 90\% and 70\% of instances respectively.
From \cref{tab:agg}, the shifted geometric means of iteration count and solve time improve by nearly 40\% and over 15\% respectively.

Since \emph{comb} computes four directions per iteration (unadjusted and TOA directions for both prediction and centering) instead of two, the per iteration times for \emph{RHS} and \emph{direc} approximately double in \cref{tab:subtime}.
The \emph{search} time increases because on average more backtracking curve search steps are needed per iteration (for \emph{curve}, the centering phase typically does not require multiple backtracking steps).
Per iteration, \emph{LHS} remains slower than the other components combined.
Hence combining the prediction and centering phases generally improves practical performance, and should be more helpful when \emph{LHS} is particularly expensive (such as when $n - p$, the side dimension of the PSD matrix we factorize during \emph{LHS}, is large; see \cref{sec:linalg}).
Furthermore, \cref{fig:comb} (right) shows that for most instances, \emph{RHS} accounts for a small proportion of the overall solve time for \emph{comb}, especially for instances that take longer to solve.
This suggests that the TOO is rarely a bottleneck for our \emph{comb} stepping procedure.

%% file: tikz/npqKhist.tex
\centering
\begin{tikzpicture}
\begin{groupplot}[
    group style = {group size=2 by 1, horizontal sep=1cm},
    width = 7cm,
    height = 5cm,
    scale only axis,
    ymin=0, 
    ybar,
    table/col sep=comma,
    typeset ticklabels with strut,
    ]
\nextgroupplot[
    xlabel = instance size ($n + p + q$), 
    xmin=0,
    xmax=5,
    xtick={0,1,2,3,4,5},
    xticklabels={$10^0$,$10^1$,$10^2$,$10^3$,$10^4$,$10^5$},
    hist={bins=20, data min=0, data max=5},
    ]
\addplot+ [] table [y index=0, y=log_npq] {csvs/inst_stats.csv};
\nextgroupplot[
    xlabel = exotic cone count ($K$),
    xmin=0,
    xmax=3.5,
    xtick={0,1,2,3},
    xticklabels={$10^0$,$10^1$,$10^2$,$10^3$},
    hist={bins=20, data min=0, data max=4},
    ]
]
\addplot+ [] table [y index=0, y=log_numcones] {csvs/inst_stats.csv};
\end{groupplot}
\end{tikzpicture}

%% file: tables/examples.tex
\centering
\begin{tabular}{@{}rrl@{}}
\toprule 
example & \# & cones in at least one instance \\
\midrule
CBLIB & 10 &
$\knn \; \kpsd \; \kltwo \; \ksqr \; \klog \; \kgpower$ 
\\
central polynomial matrix & 24 & 
$\knn \; \kpsd \; \ksqr \; \kgpower \; \krtdet \; \klog \; \ksepspec$ 
\\
classical-quantum capacity & 9 & 
$\knn \; \kpsd \; \klog \; \ksepspec$ 
\\
condition number & 6 & 
$\knn \; \kpsd \; \klmi$ 
\\
contraction analysis & 8 & 
$\kpsd \; \kmatwsos$ 
\\
convexity parameter & 7 & 
$\kpsd \; \kmatwsos$ 
\\
covariance estimation & 13 &
$\knn \; \kpsd \; \ksqr \; \kgpower \; \krtdet \; \klog \; \ksepspec$
\\
density estimation & 16 & 
$\knn \; \kpsd \; \ksqr \; \kgeom \; \klog \; \kwsos$ 
\\
discrete maximum likelihood & 7 &
$\knn \; \kpower \; \klog \; \ksepspec$
\\
D-optimal design & 16 &
$\knn \; \kpsd \; \klinf \; \kltwo \; \ksqr \; \kgeom \; \krtdet \; \klog \; \klogdet$
\\
entanglement-assisted capacity & 3 &
$\kpsd \; \ksepspec \; \kmatrelentr$
\\
experiment design & 13 & 
$\knn \; \kpsd \; \ksqr \; \kgpower \; \krtdet \; \klog \; \ksepspec$ 
\\
linear program & 3 & 
$\knn$ 
\\
Lotka-Volterra & 3 & 
$\kpsd$ 
\\
Lyapunov stability & 10 & 
$\kpsd \; \kmatsqr$ 
\\
matrix completion & 11 & 
$\knn \; \kpsd \; \ksqr \; \klspec \; \kgpower \; \kgeom \; \klog$ 
\\
matrix quadratic & 8 & 
$\kpsd \; \kmatsqr$ 
\\
matrix regression & 11 & 
$\knn \; \kpsd \; \klinf \; \kltwo \; \ksqr \; \klspec$ 
\\
maximum volume hypercube & 15 & 
$\knn \; \klinf \; \kltwo \; \ksqr \; \kgeom $ 
\\
nearest correlation matrix & 3 &
$\kmatrelentr$
\\
nearest polynomial matrix & 8 & 
$\kpsd \; \kwsos \; \kmatwsos$ 
\\
nearest PSD matrix & 28 & 
$\kpsd \; \ksppsd$ 
\\
nonparametric distribution & 10 &
$\knn \; \ksqr \; \kgeom \; \klog \; \ksepspec$
\\
norm cone polynomial & 10 & 
$\klonewsos \; \kltwowsos$ 
\\
polynomial envelope & 7 & 
$\kwsos$ 
\\
polynomial minimization & 15 & 
$\kpsd \; \kwsos$ 
\\
polynomial norm & 10 & 
$\kwsos \; \kmatwsos \; \klonewsos \; \kltwowsos$ 
\\
portfolio & 9 & 
$\knn \; \klinf \; \kltwo$ 
\\
region of attraction & 6 & 
$\kpsd \; \kwsos$ 
\\
relative entropy of entanglement & 6 &
$\kpsd \; \kmatrelentr$ 
\\
robust geometric programming & 6 & 
$\knn \; \klinf \; \klog \; \krelentr$ 
\\
semidefinite polynomial matrix & 18 & 
$\kpsd \; \kltwo \; \kmatwsos$ 
\\
shape constrained regression & 11 & 
$\knn \; \kpsd \; \klinf \; \kltwo \; \kwsos \; \kmatwsos$ 
\\
signomial minimization & 13 & 
$\knn \; \klog \; \krelentr$ 
\\
sparse LMI & 15 &
$\kpsd \; \ksppsd \; \klmi$
\\
sparse principal components & 6 & 
$\knn \; \kpsd \; \klinf$ 
\\
stability number & 6 & 
$\knn \; \kpsd \; \kdnn$ 
\\
\bottomrule
\end{tabular}

%% file: tikz/solvehist.tex
\centering
\begin{tikzpicture}[]
\begin{groupplot}[
    group style = {group size=2 by 1, horizontal sep=0.9cm},
    width = 7cm,
    height = 5cm,
    scale only axis,
    ymin=0, 
    ybar,
    table/col sep=comma,
    typeset ticklabels with strut,
    ]
\nextgroupplot[
    xlabel = iteration count, 
    xtick={0,1,2,3}, 
    xticklabels={$10^0$,$10^1$,$10^2$,$10^3$},
    xmax=3,
    xmin=0,
    hist={bins=20, data min=0, data max=3},
    ]
\addplot+ [
    red!60,
    fill opacity=0.2,
    ] table [y index=0, y=log_iters]  {csvs/basicconv.csv};
\addplot+ [
    blue!60,
    fill opacity=0.4,
    ] table [y index=0, y=log_iters]  {csvs/combconv.csv};
\legend{basic,comb}
\nextgroupplot[
    xlabel = solve time (seconds),
    xtick={-4,-2,0,2,4}, 
    xticklabels={$10^{-4}$,$10^{-2}$,$10^0$,$10^2$,$10^4$},
    xmax=4,
    xmin=-4.5,
    hist={bins=20, data min=-4, data max=4},
    ]
\addplot+ [
    red!60,
    fill opacity=0.2,
    ] table [y index=0, y=log_solve_time] {csvs/basicconv.csv};
\addplot+ [
    blue!60,
    fill opacity=0.4,
    ] table [y index=0, y=log_solve_time] {csvs/combconv.csv};
\legend{basic,comb}
\end{groupplot}
\end{tikzpicture}

%% file: tables/agg.tex
\sisetup{
table-text-alignment = right,
table-auto-round,
table-format = 3.1,
}
\centering
\begin{tabular}{
cc
SSS
S[table-format = 4]S[table-format = 4]S[table-format = 4]
}
\toprule
& & \multicolumn{3}{c}{iterations} & \multicolumn{3}{c}{solve time} \\
\cmidrule(lr){3-5} \cmidrule(lr){6-8}
step & conv & {every} & {this} & {all} & {every} & {this} & {all} \\
\midrule
basic & 371 & 101.34 & 100.93 & 102.39 & 2130.98 & 2207.10 & 2282.19 \\
prox & 369 & 64.73 & 65.28 & 67.23 & 1316.50 & 1390.14 & 1451.24 \\
TOA & 374 & 34.98 & 35.31 & 36.08 & 1014.26 & 1062.66 & 1103.01 \\
curve & 372 & 29.67 & 30.01 & 30.99 & 742.49 & 780.95 & 820.16 \\
comb & 367 & 18.31 & 18.55 & 20.02 & 623.82 & 655.71 & 706.49 \\
\bottomrule
\end{tabular}

%% file: tikz/perftotal.tex
\centering
\begin{tikzpicture}
\begin{groupplot}[
    group style = {group size=2 by 1, horizontal sep=0.9cm},
    width = 7cm,
    height = 5cm,
    scale only axis,
    xmin=0, 
    ymin=0, ymax=1,
    no markers,
    every axis plot/.append style={line width=1pt},
    legend pos = south east,
    x tick label style={/pgf/number format/.cd, fixed},
    table/col sep=comma,
    ]
\nextgroupplot[
    title = iteration count, 
    xmax=5, 
    ]
\addplot [black] table [x=x, y=y] 
{csvs/basic_vs_comb_iters.csv};
\addplot [red] table [x=x, y=y] 
{csvs/comb_vs_basic_iters.csv};
\legend{basic,comb};
\nextgroupplot[
    title = solve time,
    xmax=4, 
    ]
\addplot [black] table [x=x, y=y] 
{csvs/basic_vs_comb_solve_time.csv};
\addplot [red] table [x=x, y=y] 
{csvs/comb_vs_basic_solve_time.csv};
\legend{basic,comb};
\end{groupplot}
\end{tikzpicture}

%% file: tikz/comb.tex
\centering
\begin{tikzpicture}[]
\begin{groupplot}[
    group style = {group size=2 by 1, horizontal sep=1.1cm},
    width = 6.4cm,
    height = 5cm,
    scale only axis,
    table/col sep=comma,
    xticklabel style={/pgf/number format/fixed},
    ]
\nextgroupplot[
    xmode = log,
    ymode = log,
    xlabel = instance size ($n + p + q$),
    ylabel = solve time (seconds),
    ]
\addplot+ [only marks, mark size = 0.75pt] table [x=npq, y=solve_time] {csvs/combconv.csv};
\nextgroupplot[
    xmin = 0,
    ymode = log,
    xlabel = proportion of time in \emph{RHS},
    ]
\addplot+ [only marks, mark size = 0.75pt] table [x=prop_rhs, y=solve_time] {csvs/combconv.csv};
\end{groupplot}
\end{tikzpicture}

%% file: tables/subtime.tex
\sisetup{
table-text-alignment = right,
table-auto-round,
}
\centering
\begin{tabular}{
@{}cc
S[table-format = 2.1]
S[table-format = 3]
S[table-format = 2.2]
S[table-format = 2.1]
S[table-format = 3.1]
S[table-format = 1.2]
S[table-format = 1.2]
S[table-format = 1.2]
S[table-format = 1.2]
}
\toprule
& & & \multicolumn{4}{c}{total} & \multicolumn{4}{c}{per iteration} \\
\cmidrule(lr){4-7} \cmidrule(lr){8-11}
set & step & {init} & {LHS} & {RHS} & {direc} & {search} & {LHS} & {RHS} & {direc} & {search} \\
\midrule
\multirow{5}{*}{every}
 & basic & 29.49 & 741.38 & 1.45 & 75.59 & 125.64 & 7.73 & 0.02 & 0.81 & 1.29 \\
 & prox & 29.45 & 486.34 & 1.11 & 50.33 & 67.74 & 7.88 & 0.02 & 0.84 & 1.10 \\
 & TOA & 29.42 & 284.68 & 10.96 & 52.21 & 73.33 & 8.28 & 0.32 & 1.53 & 2.14 \\
 & curve & 29.55 & 244.38 & 9.24 & 44.60 & 33.43 & 8.33 & 0.32 & 1.53 & 1.15 \\
 & comb & 29.31 & 159.94 & 10.51 & 57.57 & 35.23 & 8.74 & 0.58 & 3.14 & 1.94 \\
\addlinespace[2pt]
\hdashline
\addlinespace[2pt]
\multirow{5}{*}{this}
 & basic & 30.33 & 784.33 & 1.48 & 78.82 & 131.78 & 8.20 & 0.02 & 0.85 & 1.36 \\
 & prox & 30.07 & 519.48 & 1.12 & 53.43 & 72.60 & 8.35 & 0.02 & 0.88 & 1.16 \\
 & TOA & 30.16 & 301.54 & 11.97 & 55.44 & 78.34 & 8.70 & 0.35 & 1.61 & 2.26 \\
 & curve & 30.52 & 260.67 & 9.99 & 47.30 & 35.10 & 8.80 & 0.34 & 1.60 & 1.20 \\
 & comb & 30.46 & 171.04 & 11.03 & 60.72 & 36.40 & 9.23 & 0.60 & 3.27 & 1.98 \\
\addlinespace[2pt]
\hdashline
\addlinespace[2pt]
\multirow{5}{*}{all}
 & basic & 31.13 & 814.44 & 1.62 & 82.65 & 134.65 & 8.52 & 0.02 & 0.91 & 1.40 \\
 & prox & 31.29 & 549.29 & 1.25 & 56.22 & 75.09 & 8.74 & 0.02 & 0.94 & 1.20 \\
 & TOA & 31.16 & 316.91 & 12.23 & 57.52 & 79.71 & 9.04 & 0.36 & 1.66 & 2.28 \\
 & curve & 31.38 & 275.95 & 10.40 & 49.63 & 37.34 & 9.17 & 0.36 & 1.68 & 1.26 \\
 & comb & 31.36 & 188.07 & 11.88 & 64.18 & 40.02 & 9.66 & 0.63 & 3.35 & 2.10 \\
\bottomrule
\end{tabular}

%% file: tikz/relimpr.tex
\centering
\begin{tikzpicture}[]
\begin{groupplot}[
    group style = {group size=2 by 1, horizontal sep=1.1cm},
    width = 6.6cm,
    height = 5cm,
    scale only axis,
    table/col sep=comma,
    xticklabel style={/pgf/number format/fixed},
    ]
\nextgroupplot[
    ylabel = relative improvement,
    xmin = 0,
    ymax = 1, 
    xlabel = iteration count,
    ]
\addplot+ [only marks, mark size = 0.75pt] table [x=iters, y=iters_impr, col sep=comma] {csvs/basiccombconv.csv};
\nextgroupplot[
    xmode = log,
    xlabel = solve time (seconds),
    ymax = 1, 
    ]
\addplot+ [only marks, mark size = 0.75pt] table [x=solve_time, y=time_impr, col sep=comma] {csvs/basiccombconv.csv};
\end{groupplot}
\end{tikzpicture}

%% file: tikz/perf.tex
\centering
\begin{tikzpicture}
\begin{groupplot}[
    group style = {group size=2 by 4, vertical sep=0.9cm, horizontal sep=1cm},
    width = 7cm,
    height = 3.8cm,
    scale only axis,
    xmin=0, 
    ymin=0, ymax=1,
    no markers,
    every axis plot/.append style={line width=1pt},
    legend pos = south east,
    x tick label style={/pgf/number format/.cd, fixed},
    title style={yshift=5pt},
    table/col sep=comma,
    ]
\nextgroupplot[
    title = iteration count, 
    xmax=2.5, 
    ]
\addplot [black] table [x=x, y=y] 
{csvs/basic_vs_prox_iters.csv};
\addplot [red] table [x=x, y=y] 
{csvs/prox_vs_basic_iters.csv};
\legend{basic,prox};
\nextgroupplot[
    title = solve time,
    xmax=2.5, 
    ]
\addplot [black] table [x=x, y=y] 
{csvs/basic_vs_prox_solve_time.csv};
\addplot [red] table [x=x, y=y] 
{csvs/prox_vs_basic_solve_time.csv};
\legend{basic,prox};
\nextgroupplot[
    xmax=2, 
    ]
\addplot [black] table [x=x, y=y] 
{csvs/prox_vs_toa_iters.csv};
\addplot [red] table [x=x, y=y] 
{csvs/toa_vs_prox_iters.csv};
\legend{prox,TOA};
\nextgroupplot[
    xmax=1.5, 
    ]
\addplot [black] table [x=x, y=y] 
{csvs/prox_vs_toa_solve_time.csv};
\addplot [red] table [x=x, y=y] 
{csvs/toa_vs_prox_solve_time.csv};
\legend{prox,TOA};
\nextgroupplot[
    xmax=1, 
    ]
\addplot [black] table [x=x, y=y] 
{csvs/toa_vs_curve_iters.csv};
\addplot [red] table [x=x, y=y] 
{csvs/curve_vs_toa_iters.csv};
\legend{TOA,curve};
\nextgroupplot[
    xmax=2, 
    ]
\addplot [black] table [x=x, y=y] 
{csvs/toa_vs_curve_solve_time.csv};
\addplot [red] table [x=x, y=y] 
{csvs/curve_vs_toa_solve_time.csv};
\legend{TOA,curve};
\nextgroupplot[
    xmax=2, 
    ]
\addplot [black] table [x=x, y=y] 
{csvs/curve_vs_comb_iters.csv};
\addplot [red] table [x=x, y=y] 
{csvs/comb_vs_curve_iters.csv};
\legend{curve,comb};
\nextgroupplot[
    xmax=2, 
    ]
\addplot [black] table [x=x, y=y] 
{csvs/curve_vs_comb_solve_time.csv};
\addplot [red] table [x=x, y=y] 
{csvs/comb_vs_curve_solve_time.csv};
\legend{curve,comb};
\end{groupplot}
\end{tikzpicture}

%% file: linalg.tex
\section{Preprocessing and solving for search directions}
\label{sec:linalg}

We discuss preprocessing and initial point finding procedures and solving structured linear systems for directions.
Although Hypatia has various alternative options for these procedures, we only describe the set of options we fix in our computational experiments in \cref{sec:testing}, to give context for these results.
These techniques are likely to be useful for other conic PDIPM implementations.

Given a conic model specified in the general primal conic form \cref{eq:prim}, we first rescale the primal and dual equality constraints \cref{eq:prim:eq,eq:dual:eq} to improve the conditioning of the affine data.
Next, we perform a QR factorization of $A'$ and check whether any primal equalities are inconsistent (terminating if so).
We use this factorization to modify $c, G, h$ and eliminate all $p$ primal equalities (removing dual variable $y$), reducing the dimension of the primal variable $x$ from $n$ to $n - p$.
Next, we perform a QR factorization of the modified $G$.
We use this factorization to check whether any dual equalities are inconsistent (terminating if so) and to remove any redundant dual equalities, further reducing the dimension of $x$.
This factorization also allows us to cheaply compute an initial $x^0$ satisfying \cref{eq:init:x}.
Since $y$ is eliminated, we do not need to solve \cref{eq:init:y} for $y^0$.

Starting from the initial interior point $\omega^0$ defined in \cref{sec:alg:cp}, we perform PDIPM iterations until the convergence conditions in \cref{sec:alg:alg} (in the preprocessed space) are met.
Finally, we re-use the two QR factorizations to lift the approximate certificate for the preprocessed model to one for the original model.
The residual norms for the lifted certificate could violate the convergence tolerances, but we have not found such violations to be significant on our benchmark instances.

During each PDIPM iteration, we solve the linear system \cref{eq:dirs} for a single LHS matrix and between one and four RHS vectors, to obtain directions vectors needed for one of the stepping procedures described in \cref{sec:alg:step}.
Instead of factorizing the large square nonsymmetric block-sparse LHS matrix, we utilize its structure to reduce the size of the factorization needed.
Some of these techniques are adapted from methods in CVXOPT (see \citet[Section 10.3]{vandenberghe2010cvxopt}).

First we eliminate $s$ and $\kappa$, yielding a square nonsymmetric system, then we eliminate $\tau$ to get a symmetric indefinite system in $x$ and $z$. 
Most interior point solvers use a sparse LDL factorization (with precomputed symbolic factorization) to solve this system.
Although Hypatia can optionally do the same, we see improved performance on our benchmark instances by further reducing the system.
After eliminating $z$, we have a (generally dense) positive definite system, which we solve via a dense Cholesky factorization. 
In terms of the original dimensions of the model before preprocessing (assuming no redundant equalities), the side dimension of this system is $n - p$.
Finally, after finding a solution to \cref{eq:dirs}, we apply several rounds of iterative refinement in working precision to improve the solution quality.

We note that this Cholesky-based system solver method does not require explicit Hessian oracles, only oracles for left-multiplication by the Hessian or inverse Hessian.
As we discuss in \cref{sec:toos} and \citet{coey2021self}, these optional oracles can be more efficient and numerically stable to compute for many exotic cones.
For cones without these oracles, Hypatia calls the explicit Hessian matrix oracle, performing a Cholesky factorization of the Hessian if necessary.
A deeper discussion of Hypatia's linear system solving techniques and optional cone oracles is outside the scope of this paper.

%% file: search.tex
\section{Efficient proximity checks}
\label{sec:search}


Recall that each stepping procedure in \cref{sec:alg:step} uses at least one backtracking search (on a line or a curve) to find a point $\omega$ satisfying an aggregate proximity condition: $\pitwo(\omega) \leq \beta_1$ for the \emph{basic} procedure in \cref{sec:alg:step:basic} or $\piinf(\omega) \leq \beta_2$ for the procedures in \crefrange{sec:alg:step:prox}{sec:alg:step:comb}. 
In \cref{sec:alg:prox}, we define $\pitwo$ and $\piinf$ in \cref{eq:proxagg:l2,eq:proxagg:linf}.
For each primitive cone $k \in \iin{\bKn}$, $0 \leq \pi_k (\omega) \leq \piinf (\omega) \leq \pitwo (\omega)$, and by \cref{lem:interior}, $\pi_k(\omega) < 1$ implies $\bs_k \in \intr \bigl( \bK_k \bigr)$ and $\bz_k \in \intr \bigl( \bK^\ast_k \bigr)$.
We use a schedule of decreasing trial values for the step parameter $\alpha$ and accept the first value that yields a candidate point satisfying the aggregate proximity condition.

Suppose at a particular iteration of the backtracking search, we have the candidate point $\omega$.
We check a sequence of increasingly expensive conditions that are necessary for the proximity condition to hold for $\omega$.
First, we verify that $\bs_k' \bz_k > 0, \forall k \in \iin{\bKn}$, which is necessary for interiority (by a strict version of the dual cone inequality \cref{eq:Kdual}).
Note that this condition implies $\mu(\omega) > 0$.
Next, we verify that $\rho_k(\omega) < \beta, \forall k \in \iin{\bKn}$, where $\rho_k(\omega)$ is:
\begin{equation}
\rho_k(\omega) \coloneqq 
\nu_k^{-1/2} \lvert \bs_k' \bz_k / \mu - \nu_k \rvert \geq 0.
\label{eq:rhoK}
\end{equation}
In \cref{lem:rho} below, we show that $\rho_k(\omega)$ is a lower bound on $\pi_k(\omega)$, so if $\rho_k(\omega) > \beta$ then then $\pi_k(\omega) > \beta$.
Computing $\rho_k$ is much cheaper than computing $\pi_k(\omega)$ as it does not require evaluating any cone oracles.

Next, we iterate over $k \in \iin{\bKn}$ to check first the primal feasibility oracle, then the optional dual feasibility oracle if implemented, and finally the proximity condition $\pi_k(\omega) < \beta$.
Before computing $\pi_k(\omega)$, we check that the gradient and Hessian oracle evaluations approximately satisfy two logarithmic homogeneity conditions \citep[Proposition 2.3.4]{nesterov1994interior}:
\begin{equation}
(g_k(\bs_k))' (H_k (\bs_k))^{-1} g_k(\bs_k)
= -\bs_k' g_k(\bs_k) 
= \nu_k.
\label{eq:lhcheck}
\end{equation}
This allows us to reject $\omega$ if the cone oracles and the proximity value $\pi_k(\omega)$ are likely to be numerically inaccurate.

\begin{lemma}
Given a point $\omega$ for which $\mu(\omega) > 0$, for each $k \in \iin{\bKn}$, $0 \leq \rho_k(\omega) \leq \pi_k(\omega)$.
\label{lem:rho}
\end{lemma}

\begin{proof}
We fix $\mu = \mu(\omega) > 0$ for convenience.
Let $f_k$ be the $\nu_k$-LHSCB for $\bK_k$, and let the conjugate of $f_k$ be $f^\ast_k$ (see \cref{eq:conj}), which is a $\nu_k$-LHSCB for $\bK_k^\ast$. 
Let $g^\ast_k \coloneqq \nabla f^\ast_k$ and $H^\ast_k \coloneqq \nabla^2 f^\ast_k$ denote the gradient and Hessian operators for $f^\ast_k$.
Using the logarithmic homogeneity properties from \citet[Proposition 2.3.4]{nesterov1994interior}, and from the definition of $\pi_k(\omega)$ in \cref{eq:proxhessk}, we have:
\begin{subequations}
\begin{align}
(\pi_k(\omega))^2
&= (\bz_k / \mu + g_k(\bs_k))' (H_k (\bs_k))^{-1} (\bz_k / \mu + g_k(\bs_k))
\\
&= \mu^{-2} \bz_k' (H_k (\bs_k))^{-1} \bz_k + 
2 \mu^{-1} \bz_k' (H_k (\bs_k))^{-1} g_k(\bs_k) +
(g_k(\bs_k))' (H_k (\bs_k))^{-1} g_k(\bs_k)
\\
&= \mu^{-2} \bz_k' (H_k (\bs_k))^{-1} \bz_k -
2 \mu^{-1} \bz_k' \bs_k + \nu_k.
\end{align}
\label{eq:rho1}
\end{subequations}
By \citet[Equation 13]{papp2017homogeneous}, $(H_k (\bs_k))^{-1} = H^\ast_k (-g_k (\bs_k))$.
Since $f^\ast_k$ is a self-concordant barrier with parameter $\nu_k$, by \citet[Equation 5.3.6]{nesterov2018lectures} we have: $( \bz_k' g^\ast_k (-g_k (\bs_k)) )^2 \leq \nu_k \bz_k' H^\ast_k (-g_k (\bs_k)) \bz_k$.
Furthermore, $g^\ast_k (-g_k (\bs_k)) = \bs_k$.
Using these facts, from \cref{eq:rho1} we have $\rho_k(\omega) \geq 0$ and:
\begin{subequations}
\begin{align}
(\pi_k(\omega))^2
&= \mu^{-2} \bz_k' H^\ast_k (-g_k (\bs_k)) \bz_k -
2 \mu^{-1} \bz_k' \bs_k + \nu_k
\\
&\geq \nu_k^{-1} \mu^{-2} ( \bz_k' \bs_k )^2 -
2 \mu^{-1} \bz_k' \bs_k + \nu_k
\\
&= \nu_k^{-1} ( \bs_k' \bz_k / \mu - \nu_k )
\\
&= (\rho_k(\omega))^2.
\end{align}
\label{eq:rho2}
\end{subequations}
Therefore, $\pi_k(\omega) \geq \rho_k(\omega) \geq 0$ for all $k \in \iin{\bKn}$.
\end{proof}

As an aside, we can use similar arguments to \cref{lem:rho} to show that $\rho_k(\omega)$ also symmetrically bounds a conjugate proximity measure $\pi_k^\ast (\omega)$, which we define as:
\begin{equation}
\pi_k^\ast (\omega)
\coloneqq \big\lVert (H_k^\ast (\bz_k))^{-1/2} (\bs_k / \mu + g_k^\ast (\bz_k)) \big\rVert
\geq \nu_k^{-1/2} \lvert \bz_k' (\bs_k / \mu + g_k^\ast (\bz_k)) \rvert
= \rho_k(\omega).
\end{equation}
In general, we cannot check whether $\pi_k^\ast (\omega) < \beta$ because as we discuss in \cref{sec:intro:alg} we do not have access to fast and numerically stable conjugate barrier oracles ($g_k^\ast$ and $H_k^\ast$).

%% file: toos.tex
\section{Computing the TOO for some exotic cones}
\label{sec:toos}

\subsection{Intersections of slices of the PSD cone}
\label{sec:toos:psdslice}

First, we consider a proper cone $\K \subset \bbR^q$ that is an inverse linear image (or slice) of the PSD cone $\bbS^{\jmath}_{\succeq}$ of side dimension $\jmath$.
Suppose:
\begin{equation}
\K \coloneqq \{ 
s \in \bbR^q : \Lambda(s) \succeq 0 
\},
\label{eq:lambdaK}
\end{equation}
where $\Lambda : \bbR^d \to \bbS^{\jmath}$ is a linear operator, with adjoint linear operator $\Lambda^\ast : \bbS^{\jmath} \to \bbR^d$.
Then the dual cone can be characterized as:
\begin{equation}
\K^\ast \coloneqq \{ 
s \in \bbR^q : \exists S \succeq 0, s = \Lambda^\ast(S) 
\}.
\label{eq:lambdaKdual}
\end{equation}
We note that for $\kpsd$ (the self-dual vectorized PSD cone), we can let $q = d = \sdim(\jmath)$, $\Lambda(s) = \mat(s)$, and $\Lambda^\ast(S) = \vect(S)$.
Given a point $s \in \bbR^q$, strict feasibility for $\K$ can be checked, for example, by attempting a Cholesky factorization $\Lambda(s) = L L'$, where $L$ is lower triangular.

For $\K$ we have the LHSCB $f(s) = -\logdet(\Lambda(s))$ with parameter $\nu = \jmath$.
Given a point $s \in \intr(\K)$, we have $\Lambda(s) \in \bbS^{\jmath}_{\succ}$ and its inverse $\Lambda^{-1}(s) \in \bbS^{\jmath}_{\succ}$.
For a direction $\delta \in \bbR^q$, for $f$ at $s$ we can write the gradient, and the Hessian and TOO applied to $\delta$, as (compare to \citet[Section 3]{papp2019sum}):
\begin{subequations}
\begin{align}
g(s)
&= -\Lambda^\ast ( \Lambda^{-1}(s) ),
\label{eq:psdslice:grad}
\\
H(s) \delta 
&= \Lambda^\ast ( \Lambda^{-1}(s) \Lambda(\delta) \Lambda^{-1}(s) ),
\label{eq:psdslice:hess}
\\
\Tau(s, \delta) 
&= \Lambda^\ast ( \Lambda^{-1}(s) \Lambda(\delta) \Lambda^{-1}(s) \Lambda(\delta) \Lambda^{-1}(s) ).
\label{eq:psdslice:too}
\end{align}
\label{eq:psdslice}
\end{subequations}
If we have, for example, a Cholesky factorization $\Lambda(s) = L L'$ (computed during the feasibility check), then the oracles in \cref{eq:psdslice} are easy to compute if $\Lambda$ and $\Lambda^\ast$ are easy to apply.
We can compute the TOO \cref{eq:psdslice:too} using the following steps:
\begin{subequations}
\begin{align}
Y &\coloneqq L^{-1} \Lambda(\delta) \Lambda^{-1}(s),
\label{eq:lamtoo:y}
\\
Z &\coloneqq Y' Y 
= \Lambda^{-1}(s) \Lambda(\delta) \Lambda^{-1}(s) \Lambda(\delta) \Lambda^{-1}(s),
\label{eq:lamtoo:z}
\\
\Tau(s, \delta) &= \Lambda^\ast ( Z ).
\label{eq:lamtoo:too}
\end{align}
\label{eq:lamtoo}
\end{subequations}
We note \cref{eq:lamtoo:y} can be computed using back-substitutions with $L$, and \cref{eq:lamtoo:z} is a simple symmetric outer product.
We use this approach to derive simple TOO procedures for $\klmi$ in \cref{sec:toos:lmi} and for $\kwsos^\ast$ and $\kmatwsos^\ast$ in \cref{sec:toos:wsos} when $r = 1$. 

Now we consider the more general case of a cone $\K$ that can be characterized as an intersection of slices of PSD cones, for example $\kwsos^\ast$ and $\kmatwsos^\ast$ when $r > 1$.
Suppose:
\begin{equation}
\K \coloneqq \{ 
s \in \bbR^q : \Lambda_l(s) \succeq 0, \forall l \in \iin{r} 
\},
\label{eq:lambdasumK}
\end{equation}
where $\Lambda_l : \bbR^d \to \bbS^{\jmath_l}$, for $l \in \iin{r}$.
Then the dual cone can be characterized as:
\begin{equation}
\K^\ast \coloneqq \bigl\{ 
s \in \bbR^q : \exists S_1, \ldots, S_r \succeq 0, s = \tsum{l \in \iin{r}} \Lambda_l^\ast(S_l) 
\bigr\}.
\label{eq:lambdasumKdual}
\end{equation}
Feasibility for $\K$ can be checked by performing $r$ Cholesky factorizations.
If we let $f_l(s) = -\logdet(\Lambda_l(s)), \forall l \in \iin{r}$, then $f(s) = \sum_{l \in \iin{r}} f_l(s)$ is an LHSCB for $\K$ with parameter $\nu = \sum_{l \in \iin{r}} \jmath_l$.
Clearly, $g(s)$, $H(s) \delta$ (and the explicit Hessian matrix), and $\Tau(s, \delta)$ can all be computed as sums over $l \in \iin{r}$ of the terms in \cref{eq:psdslice:too}.

\subsection{LMI cone}
\label{sec:toos:lmi}

We denote the inner product of $X, Y \in \bbS^s$ as $\langle X, Y \rangle = \tr(X Y) \in \bbR$, computable in order of $s^2$ time.
For $\klmi$ parametrized by $P_i \in \bbS^s, \forall i \in \iin{d}$, we define for $w \in \bbR^d$ and $W \in \bbS^s$:
\begin{subequations}
\begin{align}
\Lambda(w) 
&\coloneqq \tsum{i \in \iin{d}} w_i P_i \in \bbS^s,
\\
\Lambda^\ast(W) 
&\coloneqq ( \langle P_i, W \rangle )_{i \in \iin{d}} \in \bbR^d.
\end{align}
\end{subequations}

Our implementation uses specializations of \cref{eq:psdslice,eq:lamtoo} for $\klmi$. 
For $w \in \intr(\klmi)$ and direction $\delta \in \bbR^d$, using the Cholesky factorization $\Lambda(w) = L L'$, we compute:
\begin{subequations}
\begin{align}
Q_i
&\coloneqq L^{-1} P_i (L^{-1})' \in \bbS^s
\quad \forall i \in \iin{d},
\\
g(w) 
&= ( -\tr(Q_i) )_{i \in \iin{d}},
\\
R
&\coloneqq \tsum{j \in \iin{d}} \delta_j Q_j \in \bbS^s,
\\
H(w) \delta
&= ( \langle Q_i, R \rangle )_{i \in \iin{d}},
\\
\Tau(w, \delta)
&= ( \langle Q_i, R' R \rangle )_{i \in \iin{d}},
\label{eq:lmitoo}
\end{align}
\end{subequations}
and we compute the explicit Hessian oracle as:
\begin{equation}
(H(w))_{i,j} 
= \langle Q_i, Q_j \rangle
\quad \forall i, j \in \iin{d}.
\end{equation}
The symmetric form of $Q_i$ and the use of a symmetric outer product $R' R$ in \cref{eq:lmitoo} are beneficial for efficiency and numerical performance.

\subsection{Matrix and scalar WSOS dual cones}
\label{sec:toos:wsos}

Recall that Hypatia uses LHSCBs for $\kwsos^\ast, \kmatwsos^\ast$, because LHSCBs for $\kwsos, \kmatwsos$ with tractable oracles are not known (see \citet{kapelevich2021sum}).
Since the scalar WSOS dual cone $\kwsos^\ast$ is a special case of the matrix WSOS dual cone $\kmatwsos^\ast$ with $t = 1$, we only consider $\kmatwsos^\ast$ here.
In general, $\kmatwsos^\ast$ is an intersection of $r$ slices of $\kpsd$ (see \cref{eq:lambdasumK}), so the gradient, Hessian, and TOO oracles are all additive; for simplicity, we only consider $r = 1$ (and $s_1 = s$, $P_1 = P$) below.

To enable convenient vectorization, we define $\rho_{i,j}$ for indices $i,j \geq 1$ as:
\begin{equation}
\rho_{i,j} \coloneqq \begin{cases}
1 & \text{if } i = j,
\\
\sqrt{2} & \text{otherwise}.
\end{cases}
\label{eq:rho}
\end{equation}
For $\kmatwsos^\ast$ parametrized by $P \in \bbR^{d \times s}$ and $t \geq 1$, we define for $w \in \bbR^{\sdim(t) d}$ and $W \in \bbS^{s t}$:
\begin{subequations}
\begin{align}
\Lambda(w)
&\coloneqq \bigl[ P' \Diag \bigl( \rho^{-1}_{i,j} w_{\max(i,j),\min(i,j),:} \bigr) P \bigr]_{i,j \in \iin{t}} \in \bbS^{s t},
\\
\Lambda^\ast(W) 
&\coloneqq ( \rho_{i,j} \diag(P (W)_{i,j} P') )_{i \in \iin{t}, j \in \iin{i}} \in \bbR^{\sdim(t) d},
\end{align}
\end{subequations}
where $w = (w_{i,j,:})_{i \in \iin{t}, j \in \iin{i}}$ and $w_{i,j,:} \in \bbR^d$ is the contiguous slice of $w$ corresponding to the interpolant basis values in the $(i,j)$th (lower triangle) position,
matrix $(S)_{i,j}$ is the $(i,j)$th block in a block matrix $S$ (with blocks of equal dimension), 
and $[S_{i,j}]_{i,j \in \iin{t}}$ is the symmetric block matrix with matrix $S_{i,j}$ in the $(i,j)$th block.

We implement efficient and numerically stable specializations of the oracles in \cref{eq:psdslice,eq:lamtoo}.
Suppose we have $w \in \intr(\kmatwsos^\ast)$ and direction $\delta \in \bbR^{\sdim(t) d}$, and a Cholesky factorization $\Lambda(w) = L L'$. 
For each $i, j \in \iin{t} : i \geq j$ and $p \in \iin{d}$, we implicitly compute oracles according to:
\begin{subequations}
\begin{align}
(Q)_{i,j,p}
&\coloneqq ((L^{-1})_{i,j} P') e_p
\in \bbR^s,
\\
(g(w))_{i,j,p}
&= -\rho_{i,j} Q_{i,:,p}' Q_{:,j,p},
\\
(R)_{i,j,p}
&\coloneqq ( L^{-1} \Lambda(\delta) (L^{-1})' Q )_{i,j} e_p
\in \bbR^s,
\\
(H(w) \delta)_{i,j,p}
&= \rho_{i,j} Q_{i,:,p}' R_{:,j,p},
\\
(\Tau(w, \delta))_{i,j,p}
&= \rho_{i,j} R_{i,:,p}' R_{:,j,p}.
\end{align}
\end{subequations}
Letting $Q^2_{i,j} \coloneqq (Q' Q)_{i,j} \in \bbS^d$, we compute the Hessian oracle according to:
\begin{equation}
(H(w))_{(i,j,:), (k,l,:)}
= \tfrac{1}{2} \rho_{i,j} \rho_{k,l} \bigl(
Q^2_{i,k} \circ Q^2_{j,l} + Q^2_{i,l} \circ Q^2_{j,k}
\bigr) \in \bbS^d
\quad \forall i, j, k, l \in \iin{t},
\end{equation}
where $X \circ Y \in \bbS^d$ denotes the Hadamard (elementwise) product of $X, Y \in \bbS^d$.

\subsection{Sparse PSD cone}
\label{sec:toos:sppsd}

Let $\mathcal{S} = ((i_l, j_l))_{l \in \iin{d}}$ be a collection of row-column index pairs defining the sparsity pattern of the lower triangle of a symmetric matrix of side dimension $s$ (including all diagonal elements).
We do not require $\mathcal{S}$ to be a chordal sparsity pattern (unlike \citet{andersen2013logarithmic,burer2003semidefinite}), as this restriction is not necessary for the oracles Hypatia uses.
Note $s \leq d \leq \sdim(s)$.
For $\ksppsd$ parametrized by $\mathcal{S}$, we define $\Lambda : \bbR^d \to \bbS^s$ as the linear operator satisfying, for all $i, j \in \iin{s} : i \geq j$:
\begin{equation}
( \Lambda(w) )_{i,j} \coloneqq
\begin{cases}
\rho^{-1}_{i,j} w_l & \text{if } i = i_l = j = j_l,
\\
0 & \text{otherwise},
\end{cases}
\label{eq:sppsd:lam}
\end{equation}
where $\rho_{i,j}$ is given by \cref{eq:rho}.
Then $\Lambda^\ast$ is the vectorized projection onto $\mathcal{S}$, i.e.\ for $W \in \bbS^s$:
\begin{equation}
\Lambda^\ast(W) 
\coloneqq ( \rho_{i, j} W_{i, j} )_{(i,j) \in \mathcal{S}} \in \bbR^d.
\label{eq:sppsd:lamstar}
\end{equation}
Consider $w \in \intr(\ksppsd)$ and direction $\delta \in \bbR^d$.
The gradient \cref{eq:psdslice:grad} and Hessian product \cref{eq:psdslice:hess} for $\ksppsd$ can be computed using \citet[Algorithms 4.1 and 5.1]{andersen2013logarithmic}.
To derive the TOO, we use the fact that:
\begin{equation}
-2 \Tau (w, \delta) 
= \nabla^3 f(w) [\delta, \delta] 
= \tfrac{d^2}{d t^2} \nabla f(w + t \delta) \big\vert_{t=0}.
\label{eq:toodir}
\end{equation}

In order to succinctly describe our TOO approach as an extension of the procedures in \citet{andersen2013logarithmic}, we describe an approach based on a sparse LDL factorization of $\Lambda(w)$.
However, our current implementation in Hypatia uses a sparse Cholesky ($L L'$) factorization, which is very similar to the LDL-based approach here.
We compute the sparse Cholesky factors using Julia's SuiteSparse wrapper of CHOLMOD \citep{chen2008algorithm}.
We note that Hypatia implements a \emph{supernodal} generalization (see \citet[Section 7]{andersen2013logarithmic}) of the TOO procedure we describe below.
Before we describe the TOO procedure, we repeat useful definitions from \citet{andersen2013logarithmic}, define higher order derivative terms, and differentiate several equations that are used for the gradient and Hessian oracles.
As discussed in \cref{sec:cones}, Hypatia computes the feasibility check and gradient oracles before the TOO, and our TOO procedure reuses cached values computed for these oracles.

We define:
\begin{equation}
R \coloneqq \Lambda( \nabla f(w + t \delta) ).
\label{eq:toodirR}
\end{equation}
Let $L D L' = \Lambda(w)$ be a sparse LDL factorization, i.e.\ $L$ is a sparse unit lower triangular matrix and $D$ is a positive definite diagonal matrix.
The sparsity pattern of $L$ is associated with an \emph{elimination tree} \citep[Section 2]{andersen2013logarithmic}, and each column of $L$ corresponds to a node of this tree.
Let $I_k$ be the ordered row indices of nonzeros below the diagonal in column $k$ of $L$, and let $J_k = I_k \cup \{k\}$.
Let $\ch(i)$ denote the children of node $i$ in the tree.
For an index set $I$ let $I(i)$ denote the $i$th element.
For index sets $J \subset I$, we define $E_{I,J} \in \bbR^{\lvert I \rvert \times \lvert J \rvert}$ satisfying, $i \in \iin{\lvert I \rvert}$, $j \in \iin{\lvert J \rvert}$:
\begin{equation}
(E_{I,J})_{i,j} \coloneqq \begin{cases}
1 & \text{if } I(i) = J(j),
\\
0 & \text{otherwise.}
\end{cases}
\end{equation}
Let $U_i$ be the update matrix for node $i$ (see \citet[Equation 14]{andersen2013logarithmic}):
\begin{align}
U_i \coloneqq 
-\tsum{k \in \ch(i) \cup \{i\}} D_{k,k} L_{I_i,k} L_{I_i,k}'.
\end{align}
Let $\dotD$, $\dotL$, $\dotU$, $\dotR$ and $\ddotD$, $\ddotL$, $\ddotU$, $\ddotR$ denote the first and second derivatives of $D$, $L$, $U$, $R$ with respect to the linearization variable $t$ in \cref{eq:toodir}.
For convenience, we let:
\begin{equation}
\bar{L}_j \coloneqq \begin{bmatrix}
1 & 0 \\
-L_{I_j,j} & I
\end{bmatrix}.
\end{equation}

Suppose we have computed $\dotD$, $\dotL$, $\dotU$ according to \citet[Equation 30]{andersen2013logarithmic}.
Differentiating \citet[Equation 30]{andersen2013logarithmic} once with respect to $t$ gives:
\begin{equation}
\begin{bmatrix}
\ddotD_{j,j} 
& P_j'
\\[1ex]
P_j
& 2 D_{j,j} \dotL_{I_j,j} \dotL_{I_j,j}' + \ddotU_j
\end{bmatrix}
=
\bar{L}_j \left(
\tsum{i \in \ch(j)} E_{J_j,I_i} \ddotU_i E_{J_j,I_i}'
\right) \bar{L}_j'
,
\label{eq:sparsepsd1}
\end{equation}
where $P_j \coloneqq 2 \dotD_{j,j} \dotL_{I_j,j} + D_{j,j} \ddotL_{I_j,j}$ for convenience.
This allows us to compute $\ddotD$, $\ddotL$, $\ddotU$.
\citet[Equations 21 and 22]{andersen2013logarithmic} show that:
\begin{subequations}
\begin{align}
R_{I_j,j} 
&= -R_{I_j,I_j} L_{I_j,j},
\label{eq:SIjj}
\\
\begin{bmatrix}
R_{j,j} & R_{I_j,j}'
\\[0.8ex]
R_{I_j,j} & R_{I_j,I_j}
\end{bmatrix}
\begin{bmatrix}
1 \\
L_{I_j, j}
\end{bmatrix}
&=
\begin{bmatrix}
D_{j,j}^{-1} 
\\[0.8ex]
0
\end{bmatrix}
,
\label{eq:Sblock}
\end{align}
\end{subequations}
for each node $j$.
Differentiating \cref{eq:SIjj} once with respect to $t$ gives:
\begin{equation}
\dotR_{I_j,j} = -R_{I_j,I_i} \dotL_{I_j,j} - \dotR_{I_j,I_j} L_{I_j,j}.
\label{eq:dSIjj}
\end{equation}
Differentiating \cref{eq:Sblock} twice and substituting \cref{eq:SIjj,eq:dSIjj}, we have:
\begin{equation}
\begin{bmatrix}
\ddotR_{j,j} & \ddotR_{I_j,j}' 
\\[1ex]
\ddotR_{I_j,j} & \ddotR_{I_j,I_j}
\end{bmatrix}
=
\bar{L}_j'
\begin{bmatrix}
2 \dotD_{j,j}^2 D_{j,j}^{-3} - \ddotD_{j,j} D_{j,j}^{-2} + 2 \dotL_{I_j,j}' R_{I_j,I_j} \dotL_{I_j,j} 
& Q_j'
\\[1ex]
Q_j
& \ddotR_{I_j,I_j}
\end{bmatrix} 
\bar{L}_j
,
\label{eq:sparsepsd2}
\end{equation}
where $Q_j \coloneqq -R_{I_j,I_j} \ddotL_{I_j,j} - 2 \dotR_{I_j,I_j} \dotL_{I_j,j}$ for convenience.
This allows us to compute $\ddotR$.
Finally, by \cref{eq:toodir,eq:toodirR}, we can compute the TOO as:
\begin{equation}
-2 \Tau (w, \delta) = \Lambda^\ast \bigl( \ddotR \bigr).
\label{eq:sparsepsd:too}
\end{equation}

We now write the high-level TOO procedure.
For convenience, we let:
\begin{equation}
\Delta = \Lambda(\delta) \in \bbS^s.
\end{equation}
Following \citet{andersen2003implementing}, we define $K$ and $M$ as sparse matrices with the same structure as $L$, satisfying for all $j \in \iin{s}$:
\begin{subequations}
\begin{align}
K_{j,j} &= \dotD_{j,j}, 
\label{eq:Kjj}
\\
K_{I_j,j} &= D_{j,j} \dotL_{I_j,j}, 
\label{eq:KIj}
\\
M_{j,j} &= D_{j,j}^{-2} K_{j,j}, 
\label{eq:Mjj} 
\\
M_{I_j,j} &= D^{-1}_{j,j} R_{I_j,I_j} K_{I_j,j}. 
\label{eq:MIj}
\end{align}
\end{subequations}
The first three steps in the TOO procedure below compute $\dotD$, $\dotL$, $\dotU$, and $\dotR$ and are identical to steps in \citet[Algorithm 5.1]{andersen2013logarithmic}.
\begin{enumerate}[wide]
\item Iterate over $j \in \iin{s}$ in topological order, computing $K_{J_j,j}$ and $\dotU_j$ according to:
\begin{equation}
\begin{bmatrix}
K_{j,j} & K_{I_j,j}' \\
K_{I_j,j} & U_j'
\end{bmatrix}
=
\bar{L}_j \left(
\begin{bmatrix}
\Delta_{j,j} & \Delta_{I_j,j}' \\
\Delta_{I_j,j} & 0
\end{bmatrix}
+ \tsum{i \in \ch(j)} E_{J_j,I_i} U_i' E'_{J_j,I_i}
\right) \bar{L}_j'.
\end{equation}
\item For $j \in \iin{s}$, store $\dotD_{j,j}$ and $\dotL_{I_j,j}$ from \cref{eq:Kjj,eq:KIj}, and compute $M_{J_j,j}$ from \cref{eq:Mjj,eq:MIj}.
\item Iterate over $j \in \iin{s}$ in reverse topological order, computing $\dotR_{J_j,j}$ according to:
\begin{equation}
\begin{bmatrix}
\dotR_{j,j} & \dotR_{I_j,j}' \\
\dotR_{I_j,j} & \dotR_{I_j,I_j}
\end{bmatrix}
= \bar{L}'_j
\begin{bmatrix}
\vphantom{\dotR}
M_{j,j} & M_{I_j,j}' \\
M_{I_j,j} & \dotR_{I_j,I_j}
\end{bmatrix}
\bar{L}_j,
\end{equation}
and updating matrices $\dotR_{I_j,I_j}$ for each child $i \in \ch(j)$ of vertex $j$ according to:
\begin{equation}
\dotR_{I_i,I_i} = E_{J_j,I_i}'
\begin{bmatrix}
\dotR_{j,j} & \dotR_{I_j,j}' \\
\dotR_{I_j,j} & \dotR_{I_j,I_j}
\end{bmatrix}
E_{J_j,I_i}.
\end{equation}
\item Iterate over $j \in \iin{s}$ in topological order, computing $\ddotD_{j,j}, \ddotL_{I_j,j}, \ddotU_j$ from \cref{eq:sparsepsd1}.
\item Iterate over $j \in \iin{s}$ in reverse topological order, computing $\ddotR_{j,j}$, $\ddotR_{I_j,j}, \ddotR_{I_j,I_j}$ from \cref{eq:sparsepsd2}.
\item Compute $\Tau (w, \delta)$ using $\ddotR$ and \cref{eq:sparsepsd:too}.
\end{enumerate}

\subsection{Euclidean norm cone and Euclidean norm square cone}
\label{sec:toos:eucl}

Although $\kltwo, \ksqr \subset \bbR^q$ are inverse linear images of $\bbS^q_{\succeq}$ and hence admit LHSCBs with parameter $\nu = q$, we use the standard LHSCBs with parameter $\nu = 2$, which have a different form (see \citet[Section 2.2]{vandenberghe2010cvxopt}).
For $\kltwo, \ksqr$, the LHSCB is $f(s) = -\log(s' J s)$, where $J \in \bbS^q$ is defined according to, for $i,j \in \iin{q} : i \geq j$: 
\begin{align}
J_{i,j} & \coloneqq
\begin{cases}
1 
& \text{if } j = 1 \text{ and } \bigl( i = 1 \text{ for } \kltwo \text{ or } i = 2 \text{ for } \ksqr \bigr), 
\\[0.5ex]
-1 
& \text{if } i = j \text{ and } \bigl( i > 1 \text{ for } \kltwo \text{ or } i > 2 \text{ for } \ksqr \bigr), 
\\[0.5ex]
0 & \text{otherwise.}
\end{cases}
\end{align}
Consider $s \in \intr(\K)$ and direction $\delta \in \bbR^q$, and let $\bar{J} = (s' J s)^{-1} > 0$.
The gradient, Hessian product, and TOO oracles for $\K$ are:
\begin{subequations}
\begin{align}
g(s)
&= -2 \bar{J} J s, 
\\
H(s) \delta 
&= 2 \bar{J} ( 2 \bar{J} J s s' J \delta - J \delta ),
\\
\Tau(s, \delta) 
&= \bar{J} ( 
J s \delta' H \delta + H \delta s' J \delta - s' H \delta J \delta 
).
\end{align}
\label{eq:eucl}
\end{subequations}
These oracles are computed in order of $d$ time.
The Hessian oracle is computed in order of $d^2$ time as:
\begin{equation}
H(s) = 2 \bar{J} ( 2 \bar{J} J s s' J - J ).
\end{equation}